\documentclass[10pt,a4paper]{article}
\usepackage[margin=3.24cm]{geometry}
\usepackage{graphicx} 
\usepackage{amsfonts,amsmath,amssymb,bm,dsfont,mathrsfs,physics}
\usepackage{amssymb,amsthm}
\usepackage{multirow}
\usepackage{makecell}
\usepackage{enumitem,cases}
\setlist[description]{style=unboxed,leftmargin=.5em}
\setlist[itemize]{style=sameline,leftmargin=2em}
\usepackage{float}
\usepackage[dvipsnames]{xcolor}
\colorlet{xlinkcolor}{red!50!black}

\definecolor{refblue}{RGB}{26,13,171}

\usepackage{algorithm}[0.1]
\usepackage{algorithmic}

\usepackage{hyperref}[6.83]
\hypersetup{
	colorlinks = true,
	allcolors = refblue,
	urlcolor = BrickRed,
}

\usepackage{cleveref}

\usepackage{subfigure}

\newtheorem{definition}{Definition}

\newtheorem{lemma}{Lemma}

\newtheorem{theorem}{Theorem}

\newtheorem{remark}{Remark}
\numberwithin{equation}{section}
\numberwithin{lemma}{section}
\numberwithin{example}{section}
\numberwithin{definition}{section}
\numberwithin{assumption}{section}
\numberwithin{theorem}{section}
\numberwithin{proposition}{section}
\numberwithin{corollary}{section}
\numberwithin{remark}{section}
\DeclareMathOperator*{\argmin}{arg\,min}
\newcommand\aff{{\sf aff}}
\newcommand{\boundary}{{\sf bdry}}
\newcommand{\conv}{{\sf conv}}
\newcommand{\dist}{{\sf dist}}
\newcommand{\dom}{{\sf dom}}
\newcommand{\gr}{{\sf gph}}
\newcommand{\interior}{{\sf int}}
\newcommand{\lin}{{\sf lin}}
\renewcommand{\ker}{{\sf ker}}
\newcommand{\range}{{\sf rge}}
\newcommand{\gph}{{\sf gph}}
\newcommand\Span{{\sf span}}
\renewcommand\Re{{\mathds R}}
\newcommand{\ri}{{\rm ri}}
\newcommand{\closure}{{\rm cl}}
\newcommand{\epi}{{\rm epi}}
\newcommand\rS{{\mathbb S}}
\newcommand\rX{{\mathbb X}}
\newcommand\KKT{{\rm KKT}}
\newcommand{\ds}{\displaystyle}

\def\[{\begin{equation}}
	\def\]{\end{equation}}

\setlist[description]{style=multiline,leftmargin=2em}
\setlist[itemize]{style=standard,leftmargin=2em}
\setlist[enumerate]{style=standard,leftmargin=2.2em,itemsep=3pt}

\title{\Large\bf
	Localized Dynamic Mode Decomposition with Temporally Adaptive Segmentation \thanks{The research of this work was supported by the National Key R\&D Program of China (No. 2021YFA1001300), the National Natural Science Foundation of China (Nos. 12271150, 12471405), and the Science and Technology Innovation Program of Hunan Province (No. 2025RC3080).}}

\author{Qiuqi Li\thanks{School of Mathematics, Hunan University, Changsha 410082, China.
		({\url{qli28@hnu.edu.cn}}, {\url{chang10@hnu.edu.cn},} {\url{yyf2025@hnu.edu.cn}}).}
	\and Chang Liu\footnotemark[2]
	\and Yifei Yang \footnotemark[2]
}
\date{\today}

\begin{document}
	
	\maketitle
	
	\begin{abstract}
		Dynamic mode decomposition (DMD) is a widely used data-driven algorithm for predicting the future states of dynamical systems. However, its standard formulation often struggles with poor long-term predictive accuracy. To address this limitation, we propose a localized DMD (LDMD) framework that improves prediction performance by integrating DMD’s strong linear forecasting capabilities with time-domain segmentation techniques. In this framework, the temporal domain is segmented into multiple subintervals, within which snapshot matrices are constructed and localized predictions are performed. We first present the localized DMD method with predefined segmentation, and then explore an adaptive segmentation strategy to further enhance computational efficiency and prediction robustness. Furthermore, we conduct an error analysis that provides the upper bound of the local and global truncation error for the proposed framework. The effectiveness of LDMD is demonstrated on four benchmark problems—Burgers’, Allen–Cahn, nonlinear Schrödinger, and Maxwell’s equations. Numerical results show that LDMD significantly enhances long-term predictive accuracy while preserving high computational efficiency.


		\medskip
		\noindent 
		{\bf MSCcodes:}
		37M10, 37M99, 65P99
	\end{abstract}

	\section{Introduction}\label{sec:section_intro}
	In various engineering and applied science problems, simulations based on full-order models (FOMs) are computationally prohibitive due to the high dimensionality and complexity of the governing equations. To alleviate this burden, reduced-order modeling (ROM) techniques are widely employed \cite{Quarteroni2013re, Chak2018efficient,gu2021error}. ROM methods are broadly categorized as equation-based (intrusive) \cite{nouy2010proper, Star2021novel} or data-driven (non-intrusive) \cite{ma2021data,dave2025physics,Dynamic2010Schmid,peherstorfer2016data}, depending on whether the governing equations are explicitly used. Traditional data-driven ROM approaches include proper orthogonal decomposition - radial basis function interpolation (POD-RBF) \cite{rozza2022advanced} and machine learning (ML) \cite{Mohamed2018}. In addition, operator inference \cite{peherstorfer2016data} represents a more recent non-intrusive framework that connects purely data-driven learning with the physical structure of the system.
	
	
	Dynamic mode decomposition (DMD), as a powerful data-driven non-intrusive ROM, describes the dynamical system in an equation-free manner and can be used for prediction and control \cite{dylewsky2019dynamic, Koopman2024Ghosh, Dynamic2014Proctor, Deep2018Lusch}. Similar to operator inference, the DMD method provides a purely data-driven linear approximation to nonlinear dynamics \cite{Dynamic2010Schmid}. 
	The seminal spectral theory of the Koopman operator \cite{koopman1931hamiltonian,koopman1932dynamical} provides a transparent linear representation of nonlinear dynamics in the space of observable functions, forming the theoretical basis for the development of data-driven model reduction techniques such as DMD \cite{mezic2005spectral}. The core idea of the DMD seeks to leverage Koopman operator theory to approximate a linear operator that advances a given time-series data matrix in time, thereby enabling the representation of system dynamics in a low-dimensional subspace while extracting the dominant spatio-temporal modes.

	
	
	The DMD method is closely related to the Arnoldi method \cite{Application2011Schmid} in analyzing the characteristic patterns and dynamic behaviors of processing systems. The core of the Arnoldi method involves iteratively generating Krylov subspaces and using the orthogonalization process to construct low-dimensional projection matrices. In contrast, DMD employs the Koopman operator theory as a bridge to extract system features and spatio-temporal evolution modes \cite{Analysis2013Mezic,Spectral2009Rowley}, positioning DMD as an approximation of the Koopman operator in finite-dimensional space \cite{Koopman2016Brunton}. This connection forms a solid foundation for DMD to handle nonlinear systems, making its framework increasingly robust in theoretical \cite{An2012Duke,On2014Tu} and numerical \cite{Ergodic2017Arbabi}. Later research established the convergence of the DMD algorithm under the strong operator topology (SOT) \cite{On2018Korda}. Additionally, the Liouville operator has been proposed as an alternative to the Koopman operator to improve convergence analysis \cite{Singular2023Rosenfeld}. Furthermore, the prediction accuracy of DMD has been investigated through parabolic equations, considering both local and global truncation errors \cite{Prediction2020Lu}. 
	
	
	Standard DMD may perform poorly in long-term predictions due to its reliance on short-term dynamics. To address this limitation, a variety of time-dependent variants of DMD have been developed \cite{colbrook2024multiverse,Dynamic2022Schmid}, which enhance the characterization of complex dynamical systems from different perspectives. Multiresolution DMD (mrDMD) \cite{Multiresolution2016Kutz} extracts slow-varying trends through hierarchical temporal decomposition to capture multiscale dynamics. Windowed DMD (W-DMD) \cite{tirunagari2015windowed,zhang2019online} employs a sliding time window to adapt to non-stationary dynamical evolution. Online DMD \cite{zhang2019online} enables incremental updates for real-time streaming data tracking and accelerates the computation of W-DMD. Higher Order DMD (HODMD) \cite{Higher2017Clainche} leverages Takens' delayed embedding theorem \cite{Detecting2006Takens} to incorporate richer temporal information and historical states. DMD with memory \cite{redman2025koopman} achieves adaptive integration of mode identification and prediction by comparing Koopman spectral similarities across temporal windows.

	\color{black}

	The DMD framework strongly depends on the choice of observation functions, which becomes particularly challenging in the context of strongly nonlinear systems. Since conventional DMD performs a global linearization over the entire time domain, its ability to capture local nonlinear behaviors is limited.
	Motivated by the linearized truncation of the Taylor expansion and the time-domain decomposition method \cite{Algoritmic2014Lovric,Space-time2013Hoang}, we propose a localized DMD (LDMD) method that performs local linear approximations in the time domain to better represent nonlinear dynamics. The  LDMD method integrates the time-domain decomposition method with the DMD method to effectively address long-time-scale dynamical problems. Specifically, the complete temporal domain is decomposed into a sequence of shorter segments, with DMD applied sequentially in each segment.
	Additionally, we incorporate an adaptive segmentation strategy \cite{ji2024aaroc} with an error estimator to dynamically determine the number of time steps for each stage, keeping prediction errors controlled within a stable range. This approach achieves higher accuracy, broader applicability, and faster computation for long-term predictions.  To evaluate the performance of the proposed method, we apply it to various nonlinear dynamical systems, successfully achieving accurate approximations over the entire time domain.

	
	The outline of this paper is as follows: In section \ref{sec:section_DMD}, we briefly review the DMD method and its connection to Koopman operator theory. In section \ref{sec:section_LDMD}, we introduce LDMD with predefined segmentation that requires prior knowledge of time segmentation, and LDMD with adaptive segmentation that applies residual equations as error estimators. In section \ref{sec:section_analysis}, we provide an upper bound of the local and global truncation error of the LDMD method. In section \ref{sec:section_numerical}, we selected some representative and important numerical examples to verify that our method achieves higher accuracy and more robust prediction results than those of the standard DMD, its variants, and POD-RBF. In section \ref{sec:conclu}, we summarize the results with a discussion of advantages, challenges, and future work. 
	
	\section{The DMD}
	\label{sec:section_DMD}
	DMD is a widely recognized technique that utilizes system data measurements exclusively to approximate the underlying spatiotemporal dynamics and provide predictions. In this section, we present the general form of the dynamical system as follows:
	\begin{equation}\label{eq:DMD_1}
		\frac{\partial \mathbf{u}(t,\mathbf{x})}{\partial t}=\mathbf{f}(\mathbf{u}, D\mathbf{u}, D^2\mathbf{u},\cdots),\quad \mathbf{x}\in\Omega\subset\mathbb{R}^{N_x},\quad t\in [t_0,T],
	\end{equation}
	where $D$ denotes the difference operator, $\mathbf{u}(t,\mathbf{x})$ represents the solution vector at time $t$ and spatial coordinates $\mathbf{x}\in \Omega$, $N_x$ denotes the spatial degrees of freedom in the system. The function $\mathbf{f}\in\mathbb{C}^{N_x}$ characterizes a linear or non-linear function involving $\mathbf{u}$ and its spatial derivatives. The initial condition is $\mathbf{u}(t_0,\mathbf{x})=\mathbf{u}_0(\mathbf{x})$, while the boundary conditions may be Dirichlet, Neumann, or mixed types.
	
	Equation (\ref{eq:DMD_1}) is discretized at times $t_k = t_0 + k\Delta t$, yielding the discrete system:
	\begin{equation}\label{eq:DMD_2}
		\mathbf{u}_{k+1}=\mathbf{F}(\mathbf{u}_{k}),
	\end{equation}
	where $\mathbf{u}_k=\mathbf{u}(t_0+k\Delta t, \mathbf{x})$ with $\Delta t$ being the size of the time step, $N_t$ denotes the degrees of freedom of time in the system. $\mathcal{M}$ denotes the state space, and  $\mathbf{F}$ is a map from $\mathcal{M}$ to itself. Let $\mathbf{F}^t\colon\mathcal{M}\to\mathcal{M}$ be the flow map operator, and $\mathbf{F}^{t-t_0}$ advance the initial conditions $\mathbf{u}_0(\mathbf{x})$ to the dynamical system (\ref{eq:DMD_1}) from initial time $t_0$ to final time $t$. Therefore, trajectories evolve according to
	\begin{equation}\label{eq:DMD_7}
		\mathbf{u}(t,\mathbf{x})=\mathbf{F}^{t-t_0}(\mathbf{u}_0(\mathbf{x})).
	\end{equation}
	DMD is deeply connected to the Koopman theory. We briefly review the Koopman operator theory and the DMD algorithm in subsections \ref{sec-Koopm} and \ref{sec-DMDalg}.
	
	\subsection{Koopman operator and its mode decomposition} \label{sec-Koopm}
	The Koopman operator is a linear operator defined in the space of observation functions, which facilitates the representation of nonlinear dynamical systems in the state space as linear systems in the observation function space. For dynamical systems (\ref{eq:DMD_1}) or \eqref{eq:DMD_7}, the family of Koopman operators parameterized by the time variable $t$ can be defined as follows.
	\begin{definition}[The family of Koopman operator]\label{def:koopman}
		Let $\bm{G}(\mathcal{M})$ be an infinite dimensional observation function space for any scalar-valued observable function $g: \mathcal{M} \rightarrow \mathbb{C}$.  The family of Koopman operator $\mathcal{K}^t: \bm{G}(\mathcal{M}) \rightarrow \bm{G}(\mathcal{M})$ is defined by
		\begin{equation}\label{eq:DMD_6}
			\mathcal{K}^t g(\mathbf{u}):=g\big(\mathbf{F}^t(\mathbf{u})\big), \quad \forall g\in \bm{G}(\mathcal{M}).
		\end{equation}
	\end{definition}
	\begin{definition}[Koopman operator \cite{dynamic2016kutz}]
		The discrete-time dynamical system (\ref{eq:DMD_2}) also known as a map, evolves over time and is more general than the continuous-time formulation. The corresponding discrete-time Koopman operator $\mathcal{K}$ is
		\begin{equation}\label{eq:DMD_3}
			\mathbf{u}_{k+1}=\mathbf{F}(\mathbf{u}_{k})\Rightarrow g(\mathbf{u}_{k+1})=\mathcal{K}g(\mathbf{u}_{k}),
		\end{equation}
		where $\mathcal{K}$ is an infinite-dimensional linear operator that acts on the Hilbert space. 
	\end{definition}
	
	The spectral decomposition theory of the Koopman operator provides an explicit representation of observable functions in terms of its eigenfunctions. Let $(\lambda_i,\varphi_i)_{i=1}^{\infty}$ denote the eigenpairs of the Koopman operator $\mathcal{K}$. By performing multiple measurements of the system, we obtain a set of scalar observation functions $\mathbf{g}=\{g_{1},\ldots,g_{q}\}$. Suppose the observation function can be represented as a linear combination of the Koopman eigenfunctions, we have
	\begin{equation}\label{eq:DMD_4}
		\left.\left.\left.\mathbf{g}(\mathbf{u})=\left[
		\begin{array}
			{c}g_1(\mathbf{u}) \\
			\vdots \\
			g_q(\mathbf{u})
		\end{array}\right.\right.\right.\right]=\sum_{i=1}^{\infty}\varphi_i(\mathbf{u})
		\begin{bmatrix}
			<\varphi_i,g_1> \\
			\vdots \\
			<\varphi_i,g_q>
		\end{bmatrix}=\sum_{i=1}^{\infty}\varphi_i(\mathbf{u})\mathbf{v}_i, \quad \mathbf{v}_i=\begin{bmatrix}
			<\varphi_i,g_1> \\
			\vdots \\
			<\varphi_i,g_q>
		\end{bmatrix}.
	\end{equation}
	Based on (\ref{eq:DMD_3}) and (\ref{eq:DMD_4}), the discrete-time evolution of the observable vector $\mathbf{g}(\mathbf{\mathbf{u}}_k)$ can be expressed as
	\begin{equation*}
		\mathbf{g}(\mathbf{u}_{k})=\sum_{i=1}^{\infty}\lambda_i^k\varphi_i(\mathbf{u}_0)\mathbf{v}_i.
	\end{equation*}
	The sequence of triples $\{(\lambda_{i},\varphi_{i},\mathbf{v}_{i})\}_{i=1}^{\infty}$ constitutes the Koopman mode decomposition. In this context, the DMD eigenvalues serve as approximations to the Koopman eigenvalues $\lambda_i$, the DMD modes approximate the Koopman modes $\mathbf{v}_i$, and the amplitudes of the DMD modes approximate the values of the Koopman eigenfunctions evaluated at the initial condition. When the chosen observable functions $\mathbf{g}=\{g_{1},\ldots,g_{q}\}$ are constrained to an invariant subspace spanned by the eigenfunctions of the Koopman operator $\mathcal{K}$, they induce a finite-dimensional linear operator \cite{Koopman2016Brunton,brunton2021modern}.
	
	\subsection{The DMD algorithm} \label{sec-DMDalg}
	DMD is an equation-free data-driven method that relies solely on observational data to approximate the Koopman eigenvalues and eigenvectors. Given a sequence of snapshot data
	$\{\mathbf{y}_0,\mathbf{y}_1,\cdots,\mathbf{y}_{M}\}$, sampled at uniform time intervals, where $\mathbf{y}=\mathbf{g}(\mathbf{u})$ represents the observation function. We define the data matrices of observables $\mathbf{Y}_1$ and $\mathbf{Y}_2$ as follows:
	\begin{equation}\label{eq:DMD_5}
		\begin{gathered}
			\mathbf{Y}_1=
			\begin{bmatrix}
				| & | & & | \\
				\mathbf{y}_0 & \mathbf{y}_1 & \cdots & \mathbf{y}_{M-1} \\
				| & | && |
			\end{bmatrix},\quad
			\mathbf{Y}_2=
			\begin{bmatrix}
				| & | & & | \\
				\mathbf{y}_1 & \mathbf{y}_2 & \cdots & \mathbf{y}_{M} \\
				| & | & & |
			\end{bmatrix},
		\end{gathered}
	\end{equation}
	where $M$ is the number of snapshots. The objective of DMD is to determine the Koopman matrix $\mathbf{A} \in \mathbb{C}^{qN_x\times qN_x}$ that satisfies
	\begin{equation*}
		\mathbf{y}_{k+1}= \mathbf{A}\mathbf{y}_k,
	\end{equation*}
	where $\mathbf{A}$ serves as an approximation of the Koopman operator. The best-fit DMD matrix is obtained by solving the following optimization problem:
	\begin{equation*}
		\mathbf{A}=\underset{{A}\in\mathbb{C}^{qN_x\times qN_x}}{\arg\min}\left\|\mathbf{Y}_{2}-{A}\mathbf{Y}_{1}\right\|_{F}=\mathbf{Y}_{2}\mathbf{Y}_{1}^{\dagger},
	\end{equation*}
	where $\|\cdot\|_{F}$ represents the Frobenius norm, and $\dagger$ denotes the Moore-Penrose pseudo-inverse. 
	The DMD algorithm approximates the eigenvalues and eigenvectors of $\mathbf{A}$ in the observable space. It is very expensive to do the eigen-decomposition directly on the matrix $\mathbf{A}$.  In practical computations, we employ Algorithm \ref{alg:algorithm-DMD} to approximate the eigenvalues and eigenvectors of $\mathbf{A}$. 
	
	\begin{algorithm}
		\caption{Standard DMD}
		\renewcommand{\algorithmicrequire}{\textbf{Input:}}
		\renewcommand{\algorithmicensure}{\textbf{Output:}}
		\begin{algorithmic}[1]
			\REQUIRE Snapshots $\{\mathbf{y}_{0}, \mathbf{y}_{1}, \dots, \mathbf{y}_{M}\}$, truncated rank $r$
			\ENSURE DMD solution $\mathbf{u}^{\text{DMD}}(t,\mathbf{x})$
			
			\STATE Construct $\mathbf{Y}_1$ and $\mathbf{Y}_2$ as defined in (\ref{eq:DMD_5}).
			\STATE Perform singular value decomposition (SVD) on $\mathbf{Y}_1$: $\mathbf{Y}_1 \approx \mathbf{U} \mathbf{\Sigma} \mathbf{V}^{H}$ with \\$\mathbf{U} \in \mathbb{C}^{qN_x \times r},\mathbf{\Sigma} \in \mathbb{C}^{r \times r}, \mathbf{V} \in \mathbb{C}^{M \times r}.$
			\STATE Define the reduced-order operator:
			$
			\tilde{\mathbf{A}} = \mathbf{U}^{H} \mathbf{Y}_2 \mathbf{V} \mathbf{\Sigma}^{-1}.$
			\STATE Compute the eigenvalues and eigenvectors of $\tilde{\mathbf{A}}$: $\tilde{\mathbf{A}} \mathbf{W} = \mathbf{W} \mathbf{\Lambda}, \quad \mathbf{\Lambda} = \text{diag}(\lambda_i).$
			\STATE Compute the DMD modes:
			$
			\mathbf{\Phi} = \mathbf{U} \mathbf{W}.
			$
			\STATE Compute the future state in the observable space:
			$\mathbf{g}(\mathbf{u}^{\text{DMD}}_k) = \mathbf{\Phi} \mathbf{\Lambda}^k \mathbf{b}$ with $\mathbf{b} = \mathbf{\Phi}^{\dagger} \mathbf{g}(\mathbf{u}_0),$
			and its continuous formulation:
			$\mathbf{g}(\mathbf{u}^{\text{DMD}}(t, \mathbf{x})) = \mathbf{\Phi} \, \text{diag}(\exp(\omega t)) \mathbf{b}$ with $\omega_i = \frac{\ln(\lambda_i)}{\Delta t}.$
			\STATE Transform back to the state space:
			$\mathbf{u}^{\text{DMD}}(t, \mathbf{x}) = \mathbf{g}^{-1}(\mathbf{g}(\mathbf{u}^{\text{DMD}}(t, \mathbf{x}))),$
			where $\mathbf{g}^{-1}$ is determined in the least-squares sense if $\mathbf{g}$ is not invertible.
		\end{algorithmic}
		\label{alg:algorithm-DMD}
	\end{algorithm}
	
	\begin{remark}\label{remark_observation}
		Within the framework of Koopman theory, selecting appropriate observable functions is essential for the predictive accuracy of the DMD method. An appropriately chosen set of observables can ensure that the dynamics are well represented within a finite-dimensional invariant subspace. However, identifying such observables often requires deep insight into the underlying physical mechanisms, which is often challenging even when the governing equations are explicitly known.
	\end{remark}
	
	For systems exhibiting strong nonlinearities or oscillatory behavior, the standard DMD framework often fails to capture the intrinsic dynamics accurately.
	In the subsequent section, we propose an LDMD approach, which employs a proper temporal segmentation strategy to enhance the predictive accuracy.

	\section{Localized DMD method} \label{sec:section_LDMD}
	In the simulation of time-dependent physical systems, the computational domain often contains multiple subregions with distinct physical properties. To this end, a time-domain segmentation strategy is introduced to enable DMD to extract essential features more effectively within each temporal segment. 
	
	
	\subsection{Localized DMD with predefined segmentation}
	To implement the proposed method, the time domain $[t_0, T]$ is divided into $N$ sub-intervals based on prior knowledge of the system (\ref{eq:DMD_1}):
	\begin{equation*}
		t_0 < t_1 < \dots < t_N = T,
	\end{equation*}
	where each sub-interval is referred to as a stage. Within each stage, the initial portion of the data is used to construct the snapshot matrices, which are then employed to predict the remaining evolution. Let $\hat{\mathbf{Y}}_1^i$ and $\hat{\mathbf{Y}}_2^i$ denote the snapshot matrices for the $i$-th stage ($i = 1, 2, \dots, N$). If $n_i$ snapshots are selected in each stage, they are defined as
	\begin{equation} 
		\begin{gathered}
			\hat{\mathbf{Y}}_1^i =
			\begin{bmatrix}
				| & | & & | \\
				\mathbf{y}^i_0 & \mathbf{y}^i_1 & \dots & \mathbf{y}^i_{n_i-1} \\
				| & | & & |
			\end{bmatrix},\quad
			\hat{\mathbf{Y}}_2^i =
			\begin{bmatrix}
				| & | & & | \\
				\mathbf{y}^i_1 & \mathbf{y}^i_2 & \dots & \mathbf{y}^i_{n_i} \\
				| & | & & |
			\end{bmatrix},
		\end{gathered}
	\end{equation}
	where $\mathbf{y}_k^i$ denote the observation vectors at the $k$-th time step of stage $i$, respectively. For $i = 1$, the snapshots are taken directly from the initial dataset,
	$\{\mathbf{y}_0^1, \mathbf{y}_1^1, \dots, \mathbf{y}_{n_1}^1\} = \{\mathbf{y}_0, \mathbf{y}_1, \dots, \mathbf{y}_{n_1}\}$.
	For $i \geq 2$, the final state of stage $i-1$ serves as the initial condition for stage $i$.
	
	At the beginning of each stage, a FOM—such as the finite difference (FDM) or finite element method (FEM)—is used to generate the snapshot data, governed by
	\begin{equation} \label{eq:LDMD_2}
		\frac{\partial \mathbf{u}^{i}(t, \mathbf{x})}{\partial t} = \mathbf{f}(\mathbf{u}^{i}, D\mathbf{u}^{i}, D^2\mathbf{u}^{i}, \dots), \quad \mathbf{x} \in \Omega \subset \mathbb{R}^{N_x}, \quad t \in [t_{i-1}, t_{(i-1)'}],
	\end{equation}
	where $t_{(i-1)'} = t_{i-1} + n_i \Delta t$, $t_i=t_{i-1}+(n_i+m_i)\Delta t$, $m_i$ denotes the prediction step at $i-th$ stage. The initial condition is set as $\mathbf{u}^{i+1}(t_{i}, \mathbf{x}) = \mathbf{u}^{\text{LDMD},i}(t_{i}, \mathbf{x})$, the superscript $i$ indicates the stage to which the state $\mathbf{u}$ belongs.
	The main steps of LDMD with predefined segmentation (P-LDMD) are summarized in Algorithm \ref{alg:algorithm-DMD2}.
	
	\begin{algorithm}
		\caption{LDMD with Predefined Segmentation (P-LDMD)}
		\renewcommand{\algorithmicrequire}{\textbf{Input:}}
		\renewcommand{\algorithmicensure}{\textbf{Output:}}
		\begin{algorithmic}[1]
			\REQUIRE The number of stages $N$, snapshots $\{\mathbf{y}_0, \mathbf{y}_1, \dots, \mathbf{y}_{n_1}\}$, and the truncated rank $r$.
			\ENSURE The solution of P-LDMD $\{\mathbf{u}^{\text{LDMD},i}(t, \mathbf{x})\}_{i=1}^{N}$.
			
			\FOR{$i = 1, 2, \dots, N-1$}
			\STATE Construct the DMD model with snapshots $\{\mathbf{y}^i_0, \mathbf{y}^i_1, \dots, \mathbf{y}^i_{n_i}\}$.
			\STATE Predict $\mathbf{u}^{\text{LDMD},i}(t, \mathbf{x})$ at time interval $[t_{i-1},t_i]$.
			\STATE Compute the FOM as the initial value $\mathbf{u}^{\text{LDMD},i}(t_i, \mathbf{x})$ at time interval $[t_{i},t_{i^{'}}]$ to get $n_{i+1}$ snapshots.
			\ENDFOR
			\STATE Construct the DMD model with snapshots $\{\mathbf{y}^N_0, \mathbf{y}^N_1, \dots, \mathbf{y}^N_{n_N}\}$.
			\STATE Predict $\mathbf{u}^{\text{LDMD},N}(t, \mathbf{x})$ at time interval $[t_{N-1},t_N]$.
		\end{algorithmic}
		\label{alg:algorithm-DMD2}
	\end{algorithm}
	
	To further improve the computational efficiency of the proposed LDMD approach, an adaptive domain segmentation methodology will be developed in the next subsection.

	\subsection{Localized DMD with adaptive segmentation}\label{sec:section_opt}
	
	Since DMD demonstrates strong predictive performance for linear systems, the fundamental objective of temporal segmentation is to ensure that each resulting segment can be closely approximated by a linear system representation. When the FOM equation and its reference solution are available, we attempt to deduce the theoretically optimal segmentation strategy, designated as ``opt-LDMD".

	We consider the dynamical system (\ref{eq:DMD_1}). The right-hand term is denoted concisely as $\mathbf{f}(\mathbf{u})$. To analyze the local behavior, we perform the Taylor expansion for $\mathbf{f}(\mathbf{u})$ on $\mathbf{u}$ as follows:
	\begin{equation*}
		\mathbf{f}(\mathbf{u}+\delta \mathbf{u})=\mathbf{f}(\mathbf{u})+J(\mathbf{u})\delta \mathbf{u}+\frac{1}{2}H(\mathbf{u}+\theta\delta\mathbf{u})[\delta\mathbf{u},\delta\mathbf{u}],\quad\theta\in[0,1],
	\end{equation*}
	where $\delta\mathbf{u}$ denotes a small variation added to the solution vector $\mathbf{u}$, $J(\mathbf{u})$ is the Jacobian, and $H(\mathbf{u}+\theta\delta\mathbf{u})[\delta\mathbf{u},\delta\mathbf{u}]$ represents the bilinear operation of the Hessian tensor $H$ with the vector $\delta\mathbf{u}$. We take $\mathbf{R}(\delta\mathbf{u})=\frac{1}{2}H(\mathbf{u}+\theta\delta\mathbf{u})[\delta\mathbf{u},\delta\mathbf{u}]$ as the first-order remainder term. Therefore, our segmentation strategy is designed to control the error of the first-order remainder in the Taylor expansion. We take the $L^2$-norm of the first-order term as the metric 
	\begin{equation*}
		\|\mathbf{R}(\delta\mathbf{u})\|_2=\left\|\frac{1}{2}H(\mathbf{u}+\theta\delta\mathbf{u})[\delta\mathbf{u},\delta\mathbf{u}]\right\|_2, 
	\end{equation*}
	and construct an upper bound $\varepsilon$. Whenever this bound is exceeded, manual segmentation is performed. 
	
	However, the implementation of opt-LDMD is impractical in real applications, as its segmentation requires prior knowledge of the reference solution. In the absence of such reference data, we aim to develop an adaptive segmentation framework based on an error estimator $\Delta_k$ \cite{wirtz2014posteriori,Jiang2020Adaptive}  to approximate the linearized truncation of the Taylor expansion methodology.

	
	\color{black}
	\subsection{Evaluation criteria based on residuals}
	To guide the adaptive segmentation process, our methodology employs residual-based error evaluation as the primary criterion. The residual not only quantifies the deviation between approximate and reference solutions to evaluate prediction accuracy, but also serves as an effective mathematical tool for measuring the consistency between predicted solutions and the governing equations. This dual capability makes it both practical and reliable for error evaluation in applications where reference solutions are unavailable.
	
	Using the explicit method, we discretize the Equation (\ref{eq:DMD_1}) as:
	\begin{equation} \label{eq:residual_1}
		\frac{\mathbf{u}(t_{k+1}, \mathbf{x}) - \mathbf{u}(t_k, \mathbf{x})}{\Delta t} = \mathbf{f}(\mathbf{u}(t_k, \mathbf{x}), D\mathbf{u}(t_k, \mathbf{x}), \dots).
	\end{equation}
	Defining $\mathbf{u}(t_{k}, \mathbf{x}) = \mathbf{u}_{k}$, $\mathbf{u}(t_{k+1}, \mathbf{x}) = \mathbf{u}_{k+1}$, and $\mathbf{f}(\mathbf{u}(t_k, \mathbf{x}), D\mathbf{u}(t_k, \mathbf{x}), \dots) = \mathbf{f}_{k}$, we substitute the predicted solution $\mathbf{u}^{\text{LDMD}}$ into Equation (\ref{eq:residual_1}) to derive the residual equation:
	\begin{equation} \label{eq:residual_2}
		\mathbf{R}_k = \frac{\mathbf{u}_{k+1}^{\text{LDMD}} - \mathbf{u}_{k}^{\text{LDMD}}}{\Delta t} - \mathbf{f}_k.
	\end{equation}
	The corresponding error estimator can be defined as the norm of the residual equation:
	\begin{equation*}
		\Delta_k = \|\mathbf{R}_k\|_F=\left\|\frac{\mathbf{u}_{k+1}^{\text{LDMD}} - \mathbf{u}_{k}^{\text{LDMD}}}{\Delta t} - \mathbf{f}_k\right\|_F,
	\end{equation*}
	where $\|\cdot\|_F$ denotes the Frobenius norm.
	
	For the LDMD with adaptive segmentation (A-LDMD), prediction accuracy is controlled through the error estimator $\Delta_k$ and a predefined tolerance threshold $\epsilon$. When $\Delta_k>\epsilon$, the current prediction stage is terminated and a new stage is initiated, during which the FOM is invoked to acquire updated snapshot data. We refer to the prediction steps as $m_i$ in each stage. To reduce computational cost, $\Delta_k$ is evaluated only at every $m$ step instead of at every prediction node. Based on this foundation, we now outline the main steps in Algorithm \ref{alg:algorithm-DMD3}.
	
	We define the prediction rate $\gamma$ as the ratio of the prediction time step to the total time steps, i.e.,
	$\gamma = 1-\frac{M}{N_t}$ 
	for the standard DMD method, and 
	\[
	\gamma = \frac{\sum_{i=1}^{N} m_i}{N_t},
	\]
	for LDMD with adaptive segmentation. 
	The prediction rate $\gamma$ generally decreases as $\epsilon$ becomes smaller, indicating a functional dependence on the error threshold $\epsilon$.
	
	\begin{remark}
		The residual-based error estimator requires explicit access to the governing equations, limiting the fully data-driven applicability of the adaptive method. P-LDMD is still a data-driven method; A-LDMD is an effective approximation of the reference method.
		Data assimilation methods, such as Kalman filters or ensemble Kalman filters, can also serve as error estimators by computing the error covariance matrix of the state vector. In the future, we will focus on data-driven approaches, including machine learning and data assimilation, as alternative error estimation methods that do not rely on explicit access to governing equations, thereby enhancing applicability to complex or partially known dynamical systems.
	\end{remark}
	
	\begin{algorithm}
		\caption{LDMD with Adaptive Segmentation (A-LDMD)}
		\renewcommand{\algorithmicrequire}{\textbf{Input:}}
		\renewcommand{\algorithmicensure}{\textbf{Output:}}
		\begin{algorithmic}[1]
			\REQUIRE Given the size of the time window $m$, snapshots $\{\mathbf{y}_{0}, \mathbf{y}_{1}, \dots, \mathbf{y}_{n_1}\}$, the truncated rank $r$, and the residual allowable upper bound $\epsilon$.
			\ENSURE The solution of A-LDMD $\{\mathbf{u}^{\text{LDMD},i}(t, \mathbf{x})\}_{i=1}^{N}$.
			
			\STATE Set the stage counter $c_i=0$ with the stage $i=1$ and the residual $\Delta_k=0$.
			\
			
			\WHILE{$\sum_i n_i + \sum_i c_i m < N_t$}
			\WHILE{$\Delta_k < \epsilon$ and $\sum_i n_i + \sum_i c_i m < N_t$}
			\STATE Construct the DMD model with snapshots $\{\mathbf{y}^i_0, \mathbf{y}^i_1, \dots, \mathbf{y}^i_{n_i}\}$.
			\STATE Advance the prediction time step $m$ and calculate the residual $\Delta_k$ here.
			\STATE $c_i=c_i+1$.
			\ENDWHILE
			\STATE Compute the FOM as the initial value $\mathbf{u}^{\text{LDMD},i}(t_i, \mathbf{x})$ to get $n_{i+1}$ snapshots.
			\STATE $i=i+1$.
			\STATE Set the stage counter $c_i=0$.
			\ENDWHILE\\
			\STATE $N=i$.
		\end{algorithmic}
		\label{alg:algorithm-DMD3}
	\end{algorithm}
	
	\begin{remark}
		In the algorithm \ref{alg:algorithm-DMD3}, the truncated rank $r$ can also be selected adaptively. A common approach is to sort the indices based on the squared singular values in descending order, i.e.,
		\begin{equation*}
			\sigma_1^2>\sigma_2^2>\cdots>\sigma_l^2,
		\end{equation*}
		and determine the truncated rank $r$ such that the following condition is satisfied: 
		\begin{equation*}
			\frac{\sum_{j=1}^r \sigma^2_j}{\sum_{j=1}^l \sigma^2_j}\geq 1-\eta,
		\end{equation*}
		where $\sigma_j$ are the singular values of the snapshot matrix $\hat{\mathbf{Y}}_1^i$ and $l$ denotes the rank of the snapshot matrix $\hat{\mathbf{Y}}_1^i$. The parameter $\eta$ is a threshold for rank selection.
	\end{remark}

	\section{Error analysis}\label{sec:section_analysis}
	In this section, we perform a theoretical error analysis, establishing that the A-LDMD method delivers well-defined upper bounds for both local and global truncation errors. Since P-LDMD constitutes a specialized subclass of the A-LDMD framework, it inherently inherits analogous bounds on local and global truncation errors. 
	
	Selecting appropriate observations for different equations is often challenging, as it requires a comprehensive understanding of the underlying system. Our LDMD method addresses this issue by decomposing complex dynamical systems with large time scales into multiple smaller time-scale segments, effectively reducing the system’s nonlinear behavior within each segment. Consequently, accurate predictions can be achieved by applying identity observations $g(\mathbf{u})=\mathbf{u}$. The linearized discrete  approximation of system \eqref{eq:DMD_1} is given by
	\begin{equation}\label{analysis_1}
		\mathbf{u}_{k+1}=\mathbf{K}\mathbf{u}_{k},\quad k=0,1,\cdots,N_t-1
	\end{equation}
	with
	$\mathbf{K}=\underset{{K}\in\mathbb{C}^{qN_x\times qN_x}}{\operatorname*{\mathrm{argmin}}}\sum_{k=0}^{N_t}\|{K}\mathbf{u}_{k}-\mathbf{u}_{k+1}\|_{F}.$
	If the spectral radius of $\mathbf{K}$, denoted as $\rho(\mathbf{K})$, satisfies $\rho(\mathbf{K})<1$, then the following lemma holds:
	\begin{lemma}[\cite{Prediction2020Lu}]\label{lem:initial_control}
		For the dynamical system (\ref{eq:DMD_1}), subsequent DMD prediction results remain bounded by the initial values:
		\begin{equation}\label{eq:analysis_1}
			||\mathbf{u}_{k+1}||^2_2<||\mathbf{u}_{0}||^2_2.
		\end{equation} 
	\end{lemma}
	Compared with Lemma $3.1$ in \cite{Prediction2020Lu}, our formulation incorporates the source term in \eqref{eq:DMD_1}, ensuring that subsequent prediction results remain constrained by the initial values, as expressed in \eqref{eq:analysis_1}.
	
	\begin{lemma}[\cite{Prediction2020Lu}]\label{lem:local_error}
		Define the local truncation error
		\begin{equation*}
			\mathbf{\tau}^{\text{DMD}}_{k}=\mathbf{u}_{k}-\mathbf{u}^{\text{DMD}}_k(\mathbf{u}_{k-1}).
		\end{equation*}
		Then, for any $k\geq M$,
		\begin{equation*}
			||\mathbf{\tau}^{\text{DMD}}_k||^2_2\leq \varepsilon_M,
		\end{equation*}
		where $\varepsilon_M$ is a constant dependent only on the number of snapshots $M$.
	\end{lemma}
	
	From Lemma \ref{lem:initial_control} and Lemma \ref{lem:local_error}, we can deduce the local truncation error bound of our LDMD method.
	
	\begin{theorem}
		\label{thm:local_error}
		Define the local truncation error of the A-LDMD
		\begin{equation*}
			\mathbf{\tau}^{\text{LDMD}}_k=\mathbf{u}_k-\mathbf{u}^{\text{LDMD}}_k(\mathbf{u}_{k-1}).
		\end{equation*}
		Then, for any $k\in [\sum_{j=1}^{i} n_j+\sum^{i-1}_{j=0}c_j m+1,\sum_{j=1}^{i} n_j+\sum^{i}_{j=0}c_j m]$ with $c_0=0$, where the superscript $i$ denotes the stage index, and the constant $\varepsilon_{n_i}^i$ is a constant depending only on the number of snapshots $n_i$ at each stage.
	\end{theorem}
	
	\begin{proof}
		Since A-LDMD performs an independent DMD algorithm at each stage, we generate the different Koopman matrices $\mathbf{K}_i$ for each stage. Therefore, according to Lemma \ref{lem:local_error}, we have the local truncation error
		\begin{equation*}
			||\mathbf{\tau}^i_k||^2_2\leq \varepsilon_{n_i}^i, \quad k\in [\sum_{j=1}^{i} n_j+\sum^{i-1}_{j=1}c_j m+1,\sum_{j=1}^{i} n_j+\sum^{i}_{j=1}c_j m].
		\end{equation*}
		
		According to the proof of the Lemma \ref{lem:local_error}  and the proof of Theorem 3.3 in \cite{Prediction2020Lu}, we have
		\begin{equation*}
			\varepsilon_{n_i}^i=||C_{n_i}^i||^2_2||\mathbf{u}_0||_2,
		\end{equation*}
		where $||C_{n_i}^i||_2^2$ is the upper bound of
		\begin{equation*}
			(1+\delta)\left\|\underset{\mathbf{A}\in\mathbb{C}^{qN_x\times qN_x}}{\operatorname*{\mathrm{argmin}}}\sum_{k=0}^{n_i+c_i m-1}||\mathbf{A}\mathbf{u}_k-\mathbf{u}_{k+1}||_2^2-\underset{\mathbf{A}\in\mathbb{C}^{qN_x\times qN_x}}{\operatorname*{\mathrm{argmin}}}\sum_{k=0}^{n_i-1}||\mathbf{A}\mathbf{u}_k-\mathbf{u}_{k+1}||_2^2\right\|^2_2,
		\end{equation*}
		where $\delta$ is a constant such that $0<\delta \ll 1$.
		
		Thus, the local truncation error bound for A-LDMD is determined by the maximum error bound over all stages:
		\begin{align*}
			&||\mathbf{\tau}^{\text{LDMD}}_k||^2_2\leq\max_{i}\{||\tau^{i}_k||_2^2\}\leq \max_{i}\{\varepsilon_{n_i}^i\},\\ & k\in [\sum_{j=1}^{i}n_j+\sum^{i-1}_{j=0}c_j m+1,\sum_{j=1}^{i}n_j+\sum^{i}_{j=0}c_j m].
		\end{align*}
	\end{proof}
	
	In short time intervals, the nonlinearity of complex systems is relatively weak. According to Koopman operator theory, the DMD algorithm can more effectively extract feature information from the system dynamics described by (\ref{eq:DMD_1}) for short-term predictions. As a result, the Koopman matrix provides a more precise representation of the system's dynamics over short time scales, leading to tighter error bounds in the local truncation error. Thus, we can reasonably deduce
	\[\max_{i}\{\varepsilon_{n_i}^i\}\leq\varepsilon_M,\]
	and
	\[||\mathbf{\tau}^{\text{LDMD}}_k||^2_2\leq||\mathbf{\tau}^{\text{DMD}}_k||^2_2.\]
	
	\begin{lemma}[\cite{Prediction2020Lu}]\label{lem:global_error}
		Define the global truncation error
		\begin{equation*}
			\mathbf{e}^{\text{DMD}}_k=\mathbf{u}_k-\mathbf{u}^{\text{DMD}}_k.
		\end{equation*}
		Then, for any $k\geq M$,
		\begin{equation}\label{analysis_4}    ||\mathbf{e}^{\text{DMD}}_k||_2<||\mathbf{\Phi}||_2||\mathbf{\Phi}^{-1}||_2[||\mathbf{e}_M||_2+(k-M)\varepsilon_M],
		\end{equation}
		where $||\mathbf{e}_M||_2$ is fixed and minimal \cite{Dynamic2010Schmid}.
	\end{lemma}
	
	Assume that the dynamical system (\ref{eq:DMD_1}) is stable and well-posed, then the solution depends continuously on the initial conditions as Lemma \ref{lem:initial_depend_continuously}.
	
	\begin{lemma}
		\label{lem:initial_depend_continuously}
		\begin{equation}\label{eq:analysis_2}
			||\mathbf{u}(t,\mathbf{x})-\mathbf{w}(t,\mathbf{x})||_2\leq L||\mathbf{u}_0(\mathbf{x})-\mathbf{w}_0(\mathbf{x})||_2,\quad \forall t\in [t_0,T],
		\end{equation}
		where $\mathbf{u}(t,\mathbf{x}),\mathbf{w}(t,\mathbf{x})$ are solutions of (\ref{eq:DMD_1}) corresponding to different initial conditions $\mathbf{u}_0(\mathbf{x})$, $\mathbf{w}_0(\mathbf{x})$, respectively, and $L$ is a constant.
	\end{lemma}
	
	From Lemma \ref{lem:global_error} and Lemma \ref{lem:initial_depend_continuously}, we can deduce the global truncation error bound of LDMD.
	
	\begin{theorem}\label{thm:global_error}
		Define the global truncation error of the A-LDMD as
		\begin{equation*}
			\mathbf{e}^{\text{LDMD}}_k=\mathbf{u}_k-\mathbf{u}^{\text{LDMD}}_k.
		\end{equation*}
		Then, for any $k\in [\sum_{j=1}^i n_j+\sum^{i-1}_{j=0}c_j m+1,\sum_{j=1}^i n_j+\sum^{i}_{j=0}c_j m]$ with $c_0=0$,  the following upper bound holds:
		\begin{equation*}
			||\mathbf{e}^{\text{LDMD},i}_k||_2<\sum_{p=1}^{i}L^{i-p}||\mathbf{\Phi}_{p}||_2||\mathbf{\Phi}_{p}^{-1}||_2(||\mathbf{e}^{p}_{n_{p}}||_2+k^p\varepsilon_{n_p}^p),
		\end{equation*}
		where $||\mathbf{e}_{n_i}^i||_2$ and $||\mathbf{e}_{n_p}^p||_2$ are fixed and minimal, $k^p$ represents the number of predicted time steps in the $p$-th stage.
	\end{theorem} 
	\begin{proof}
		The global truncation error of the A-LDMD consists of two components:
		\begin{itemize}
			\item  The DMD prediction error at the current stage, which follows the error bound described in Lemma \ref{lem:global_error}.
			\item The error propagation from previous stages, due to the reliance on previous stage predictions as initial values. 
		\end{itemize}
		To formalize the error propagation, we define the accumulated error from the previous stage as:
		\begin{equation*}
			||\mathbf{e}^{\text{LDMD},i-1}_k||_2:=||\mathbf{E}^{\text{LDMD},i-1}||_2,
		\end{equation*}
		where $\mathbf{E}^{\text{LDMD},i-1}$ represents the LDMD prediction error at the end of stage $i-1$, serving as the initial condition for stage $i$.
		
		Applying Lemma \ref{lem:global_error} and Lemma \ref{lem:initial_depend_continuously}, we establish the following bound for the global truncation error at stage $i$:
		\begin{equation}\label{eq:global}
			||\mathbf{e}^{\text{LDMD},i}_k||_2<||\mathbf{\Phi}_i||_2||\mathbf{\Phi}_i^{-1}||_2(||\mathbf{e}_{n_i}^i||_2+k^i\varepsilon_{n_i}^i)+L||\mathbf{E}^{\text{LDMD},i-1}||_2,
		\end{equation}
		where
		\begin{equation*}
			k^i=(k-\sum_{j=1}^i n_j-\sum^{i-1}_{j=0}c_j m),
			\quad k\in [\sum_{j=1}^i n_j+\sum^{i-1}_{j=0}c_j m+1,\sum_{j=1}^i n_j+\sum^{i}_{j=0}c_j m].
		\end{equation*}
		Here, $||\mathbf{e}_{n_i}^i||$ represents the error in the DMD reconstruction stage, which is fixed and minimal. $\mathbf{\Phi}_i$ represents the mode approximation of the Koopman matrix $\mathbf{K}_i$. Additionally, the initial error is zero in the first stage, i.e. $||\mathbf{E}^{\text{LDMD},0}||_2=0$.
		
		Performing the recursion in (\ref{eq:global}) iteratively from $i=1$ to $i=N$, we obtain
		
		\begin{equation*}
			||\mathbf{e}^{\text{LDMD},i}_k||_2<\sum_{p=1}^{i}L^{i-p}||\mathbf{\Phi}_{p}||_2||\mathbf{\Phi}_{p}^{-1}||_2(||\mathbf{e}^{p}_{n_{p}}||_2+k^p\varepsilon_{n_p}^p),
		\end{equation*}
		where $k^p=c_p m$ with $p=1,2,\cdots,i-1$. This completes the proof.
	\end{proof}

	\section{Numerical results}\label{sec:section_numerical}
	In this section, we perform four numerical experiments on complex systems.  Under the identical observable function, SVD truncation, and prediction rate, we exhibit that our LDMD method achieves better accuracy than the standard DMD method, POD-RBF, and time-dependent variants of DMD (mrDMD or HODMD).
As noted in Remark \ref{remark_observation}, selecting appropriate observables generally requires domain expertise; therefore, we set 
$g(\mathbf{u})=\mathbf{u}$ for the first three cases (i.e., using the identity mapping).
~In the algorithm implementation process, we evaluate whether the residual exceeds the predefined threshold at every $m$-th step. The $L^2$ relative error (RE) at time $t_k$ is defined as:
\begin{equation*}
	\text{RE}=\frac{||\mathbf{\hat{u}}_k-\mathbf{u}^{ref}_{k}||_2}{||\mathbf{u}^{ref}_{k}||_2},
\end{equation*}
where $\mathbf{\hat{u}}_k$ denotes the reconstructed or predicted solution, and $\mathbf{u}^{\text{ref}}_k$ is the reference solution. The mean $L^2$ relative error (MRE) is given by:
\begin{equation*}
	\text{MRE}=\frac{\sum_{k=1}^{N_t}\text{RE}}{N_t},
\end{equation*}
where $N_t$ is the number of time steps.

\subsection{Burgers’ equation}
We consider the one-dimensional Burgers' equation with Dirichlet boundary conditions, a canonical nonlinear model for shock waves and turbulence.
The following equation can describe it,
\begin{equation}\label{eq:numerical_1}
	\begin{cases}
		{\frac{\partial \mathbf{u}}{\partial t}=-\mathbf{u}\frac{\partial \mathbf{u}}{\partial x}+\mu\frac{\partial^{2}\mathbf{u}}{\partial x^{2}}},\quad \mathbf{x}\in[-L,L],\quad t\in[0,T], 
		\\ \mathbf{u}(0,\mathbf{x})=-\sin(\pi \mathbf{x}),
		\\ \mathbf{u}(t,-L)=\mathbf{u}(t,L)=0,
	\end{cases}
\end{equation}
where $L=1$, $T=1$ and the viscosity coefficient $\mu=0.01$. 
We discretize the spatial domain into $N_x=500$ intervals and the temporal domain into $N_t=2000$ steps. Then, we set
\[
r=20,\quad \epsilon=5\times 10^{-5},\quad n_1=300,\quad m=50.
\]
Therefore, with the number of stages $N=14$, the prediction rate is $\gamma=50\%$, corresponding to $M=1000$ in the standard DMD method. Figure \ref{fig:Burgers1} shows that both the standard DMD and A-LDMD are able to approximate the Burgers’ equation well.

\begin{figure}[htbp]
	\centering
	\subfigure{\includegraphics[width=1\textwidth]{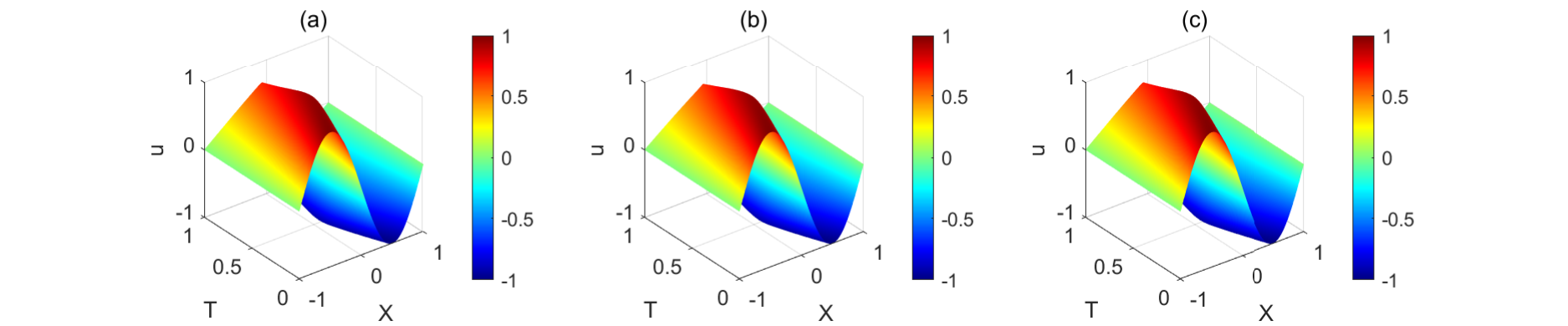}}
	\caption{The solution $\mathbf{u}(t,\mathbf{x})$ of the Burgers’ equation: (a) the reference solution, (b) the standard DMD solution, and (c) the A-LDMD solution.}
	\label{fig:Burgers1}
\end{figure}

\begin{figure}[htbp]
	\centering
	\subfigure{\includegraphics[width=0.4\textwidth,height=1.6in]{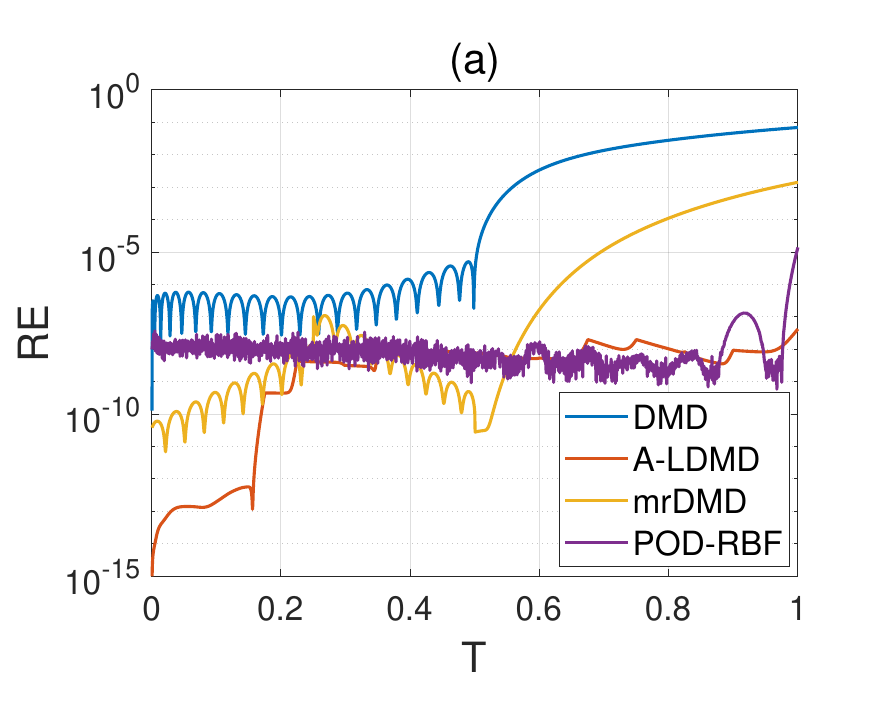}}
	\subfigure{\includegraphics[width=0.4\textwidth,height=1.6in]{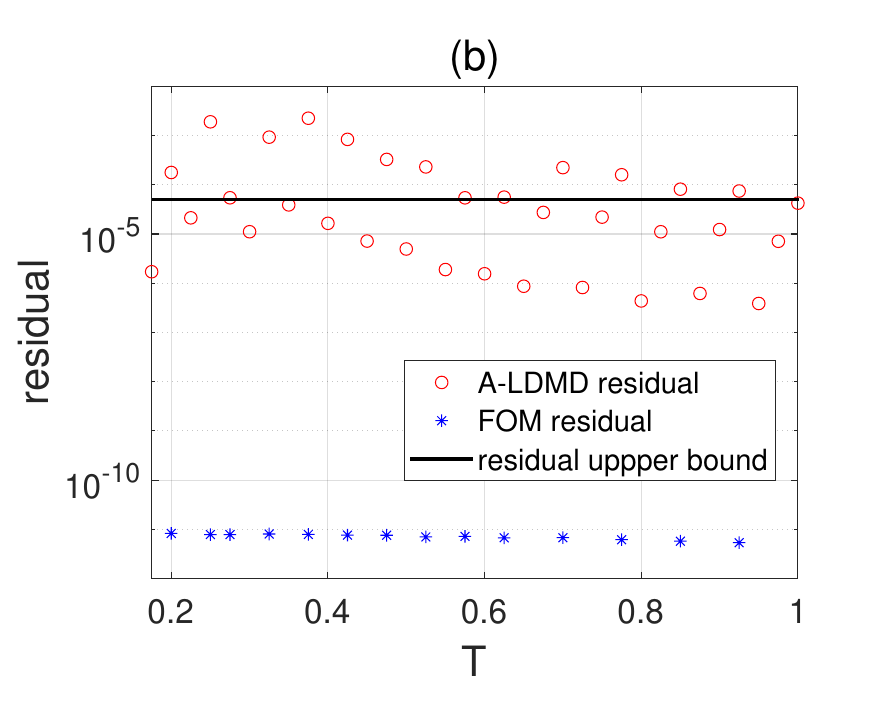}}
	\caption{(a) The $L^2$ relative error of the Burgers' equation for DMD, A-LDMD, mrDMD, and POD-RBF at $\gamma=50\%$. (b) Residual of the A-LDMD solution.}
	\label{fig:Burgers2}
\end{figure}


For a fair comparison and to highlight the advantages of our time-domain segmentation, we consider two representative methods: level-$3$ mrDMD and POD-RBF, with the relevant results shown in Figure \ref{fig:Burgers2} (a). POD-RBF snapshots are taken at the same time instances as A-LDMD, and the POD truncation number is set to $20$ for consistency.




Figure \ref{fig:Burgers2} (a) shows that the computational accuracy of all other approaches is superior to the standard DMD method. It reveals that  A-LDMD achieves the highest reconstruction accuracy, with prediction RE remaining stable between $10^{-10}$ and $10^{-7}$;
mrDMD achieves better reconstruction and short-term prediction but suffers from rapid error growth and high computational cost;
POD-RBF achieves comparable prediction accuracy to A-LDMD but exhibits noticeable fluctuations and faster error growth in the final time intervals. Table \ref{tab:burgers_4} presents the CPU time and MRE, showing that A-LDMD achieves noticeable accuracy improvement with minimal computational cost. This suggests that A-LDMD exhibits superior performance compared to the other methods


Figure \ref{fig:Burgers2} (b) illustrates the A-LDMD time-domain segmentation based on solution residuals. Whenever the residual exceeds the prescribed upper bound, the FOM is adopted to correct the corresponding portion of the prediction. Consequently, the residual equation is applied twice at these correction points, yielding two residual values. As expected, the residual remains relatively large during the prediction steps and grows with each stage of prediction until it surpasses the threshold. Once the FOM correction is applied, the residual is significantly reduced.


\begin{table}[h]
	\centering
	\caption{Comparison of CPU time and the mean $L^2$ relative error of the Burgers' equation for FOM, DMD, A-LDMD, mrDMD, and POD-RBF at $\gamma=50\%$.}
	\label{tab:burgers_4}
	\setlength{\extrarowheight}{2pt}
	\begin{tabular}{c|c|c}
		\Xhline{1pt}
		Model & CPU time(s) & MRE \\
		\hline
		FOM & 3.6716 & / \\
		\hline
		DMD & 1.8933 & $0.0125$ \\
		\hline
		A-LDMD & 2.0958 & $6.4081\times 10^{-9}$ \\
		\hline
		mrDMD & 2.2297 & $1.2127\times 10^{-4}$ \\
		\hline
		POD-RBF & 1.9047 & $8.4582\times10^{-8}$ \\
		\Xhline{1pt}
	\end{tabular}
\end{table}

\color{black}
\subsubsection{Different prediction rates}
To verify the impact of the prediction rate on prediction performance, we compare the prediction performance of A-LDMD and the standard DMD method at different prediction rates ($\gamma=40\%, 50\%, 60\%$). For $\gamma = 50\%$, we maintain the parameter configuration as above. For $\gamma = 40\%$, we set
\[
r=20,\quad \epsilon=10^{-5},\quad n_1=400,\quad m=50.
\]
{This corresponds to $N = 14$ stages}, which is equivalent to $M = 1200$ in the standard DMD method. For $\gamma = 60\%$, we set
\[
r = 15, \quad \epsilon = 10^{-3}, \quad n_1 = 200, \quad m = 50.
\]
{This corresponds to $N = 13$ stages}, which is equivalent to $M = 800$ in the standard DMD method.

\begin{figure}[htbp]
	\centering
	\subfigure{\includegraphics[width=0.32\textwidth,height=1.28in]{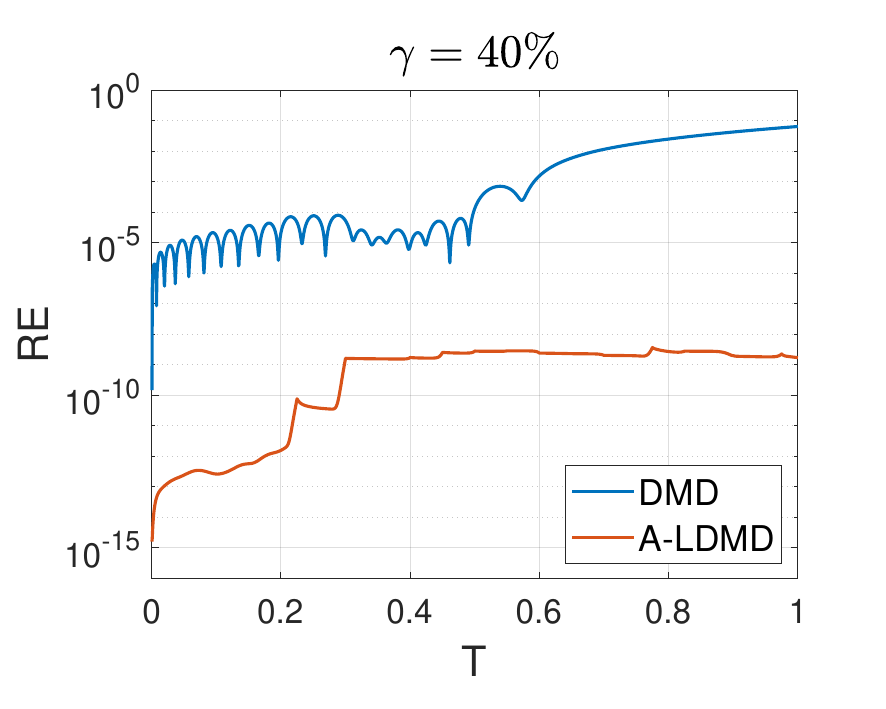}}
	\subfigure{\includegraphics[width=0.32\textwidth,height=1.28in]{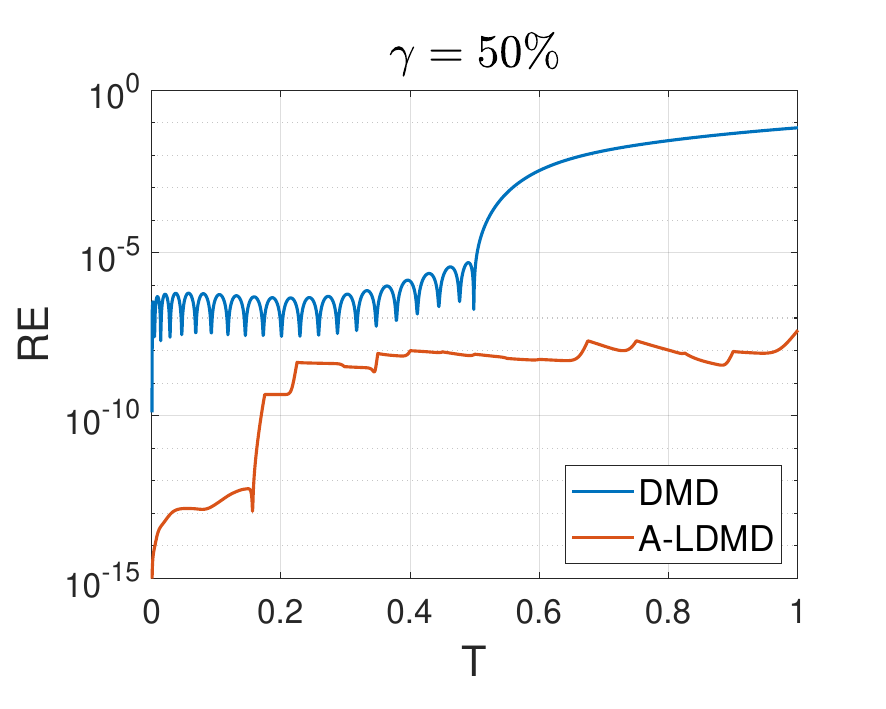}}
	\subfigure{\includegraphics[width=0.32\textwidth,height=1.28in]{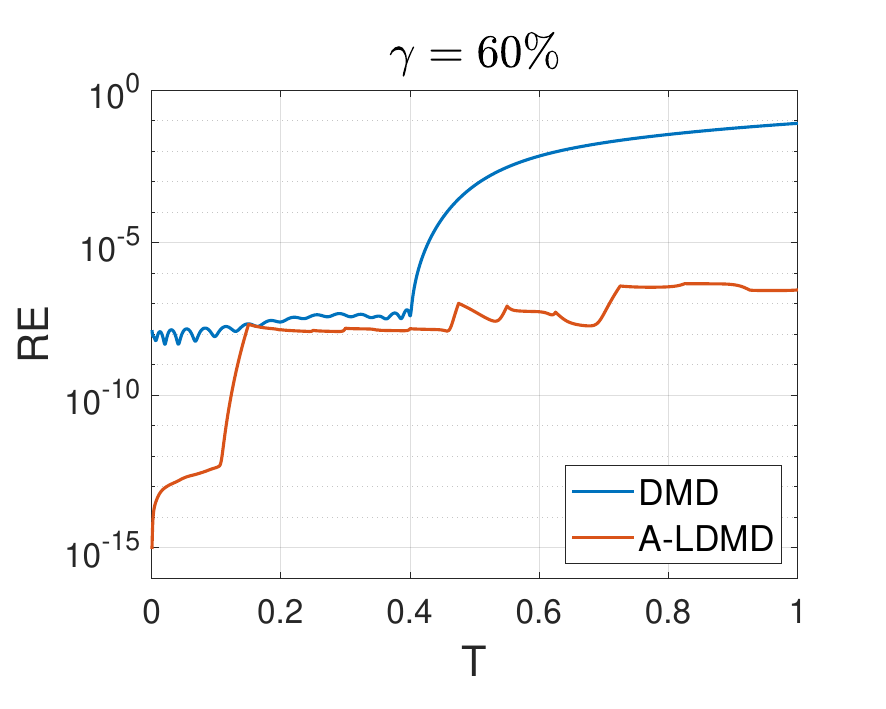}}
	\caption{The $L^2$ relative error across different prediction rates. From left to right, prediction rates $\gamma$ are $40\%,~50\%,~60\%$ respectively.}
	\label{fig:Burgers3}
\end{figure}

Figure \ref{fig:Burgers3} illustrates that A-LDMD consistently outperforms the standard DMD method in different prediction rates. Notably, smaller prediction rates yield improved accuracy, and the performance improvement of A-LDMD becomes more pronounced under a well-chosen residual error bound and time-window size are selected. 

\begin{table}[h]
\centering
\caption{Comparison of CPU time and the mean $L^2$ relative error of the Burgers' equation for FOM, DMD, and A-LDMD at different prediction rates.}
\label{tab:Burgers}
\setlength{\extrarowheight}{2pt}
\begin{tabular}{c|c|c|c}
	\Xhline{1pt}
	\multicolumn{2}{c|}{Model} & CPU time(s) & MRE \\
	\hline
	\multicolumn{2}{c|}{FOM} & 3.6716 & / \\
	\hline
	\multirow{2}{*}{$\gamma=40\%$} & DMD & 2.2748 & 0.0111 \\ 
	\cline{2-4}
	& A-LDMD & 2.4081 & $1.5760\times 10^{-9}$ \\ 
	\hline
	\multirow{2}{*}{$\gamma=50\%$} & DMD & 1.8933 & 0.0125 \\ 
	\cline{2-4}
	& A-LDMD & 2.0958 & $6.4081\times 10^{-9}$ \\ 
	\hline
	\multirow{2}{*}{$\gamma=60\%$} & DMD & 1.4352 & 0.0157 \\ 
	\cline{2-4}
	& A-LDMD & 1.6132 & $1.1922\times 10^{-7}$ \\ 
	\Xhline{1pt}
\end{tabular}
\end{table}
Table \ref{tab:Burgers} shows that the MRE of both DMD and A-LDMD improves as the prediction rate decreases, at the cost of higher computational expense due to additional FOM data. A-LDMD achieves significantly higher accuracy than DMD, with only a slight increase in CPU time.


\subsubsection{Different residual error thresholds}
To assess the influence of the residual error threshold on predictive performance, we compare the results of A-LDMD with those of the standard DMD method under varying residual error thresholds, while keeping all other settings consistent with the case of $\gamma=50\%$.

\begin{table}[h]
\centering
\caption{Comparison of CPU time and the mean $L^2$ relative error of the Burgers' equation for A-LDMD in different residual error thresholds.}
\label{tab:burgers_2}
\setlength{\extrarowheight}{2pt}
\begin{tabular}{c|c|c}
	\Xhline{1pt}
	Residual error threshold & CPU time(s) & MRE \\
	\hline
	$10^{-2}(\gamma=65\%)$ & 1.6092 & $3.6445\times10^{-6}$ \\
	\hline
	$10^{-4}(\gamma=55\%)$ & 1.8677 & $2.1019\times10^{-8}$ \\
	\hline
	$5\times10^{-5}(\gamma=50\%)$ & 2.0958 & $6.4081\times 10^{-9}$ \\
	\Xhline{1pt}
\end{tabular}
\end{table}

\begin{figure}[htbp]
\centering
\subfigure{\includegraphics[width=0.5\textwidth,height=2.0in]{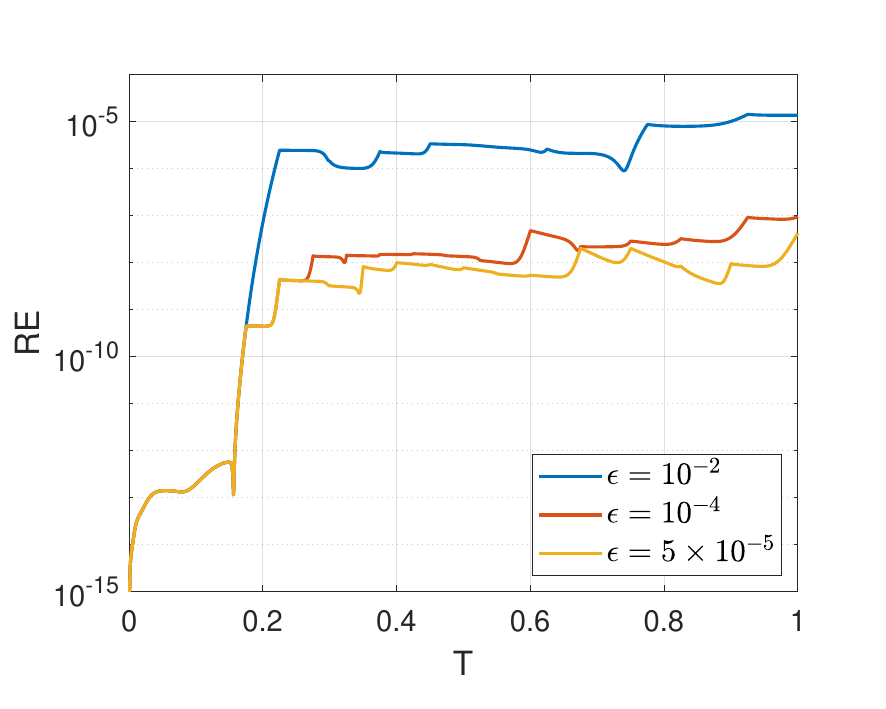}}
\caption{The $L^2$ relative error across different residual error thresholds.}
\label{fig:Burgers4}
\end{figure}

Figure \ref{fig:Burgers4} and Table \ref{tab:burgers_2} demonstrate a clear positive correlation between the residual error threshold $\epsilon$, the prediction rate $\gamma$, and the prediction accuracy, while showing an inverse correlation between $\epsilon$ and the required CPU time. Therefore, with other settings kept unchanged, variations in the residual error thresholds primarily affect the prediction rate and the location of the adaptive segmentation, thereby impacting the prediction performance. 
\begin{figure}[htbp]
\centering
\subfigure{\includegraphics[width=0.8\textwidth]{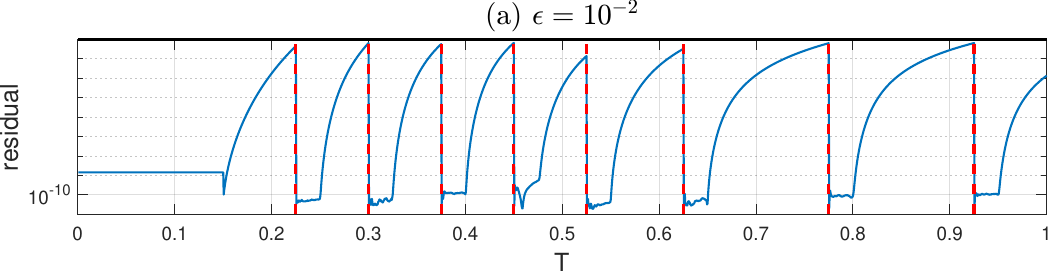}}
\subfigure{\includegraphics[width=0.8\textwidth]{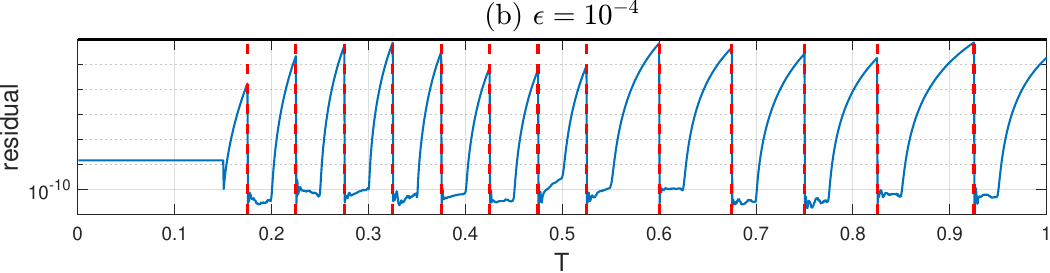}}
\subfigure{\includegraphics[width=0.8\textwidth]{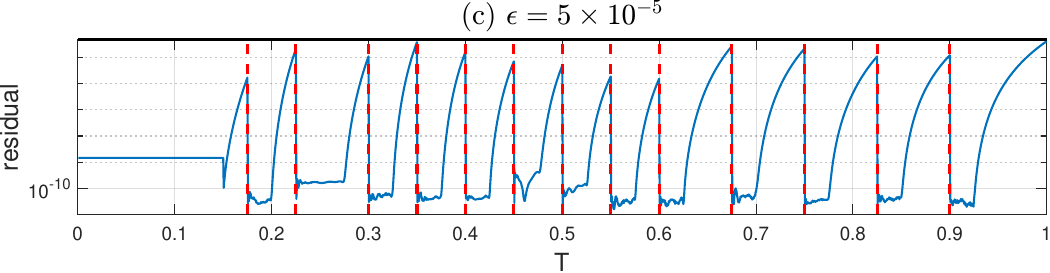}}
\caption{Comparison of the number and location of segments under different residual error thresholds $\epsilon$: (a) $\epsilon=10^{-2}$, (b) $\epsilon=10^{-4}$, and (c) $\epsilon=5\times10^{-5}$. The red dashed lines indicate the locations of the segmentation boundaries.}
\label{fig:Burgers8}
\end{figure}

Figure \ref{fig:Burgers8} illustrates that the residual error threshold $\epsilon$ directly determines the number and location of segments, thereby controlling the number of predictive steps per stage and the overall prediction rate $\gamma$. From Figure \ref{fig:Burgers4}, we can also find that the prediction error with different residual error thresholds $\epsilon$ consistently follows the same pattern: a slight initial increase followed by rapid stabilization at a low level. Moreover, A-LDMD consistently outperforms standard DMD in accuracy with the same prediction rate.

\subsubsection{Comparison of theoretically optimal and adaptive segmentation} 
Related to the theoretically optimal segmentation strategy discussed in Section \ref{sec:section_opt}, we consider here the semi-discrete conservative form of the right-hand side of Equation (\ref{eq:numerical_1}). 
\begin{equation*}
\mathbf{f}(\mathbf{u})=-\frac{1}{2}\mathbf{A}_1(\mathbf{u}\circ\mathbf{u})+\mu\mathbf{A}_2\mathbf{u},
\end{equation*}
where $\mathbf{u}\circ\mathbf{u}=[u_1^2,\cdots,u_{N_x}^2]^T$, $\mathbf{A_1}$, $\mathbf{A_2}$ represent the difference matrices for the second-order term and the first-order term, respectively. We perform the Taylor expansion on $\mathbf{u}$, then the first-order remainder term is $\mathbf{R}(\delta\mathbf{u})=-\frac{1}{2}\mathbf{A}_1(\delta\mathbf{u}\circ\delta\mathbf{u})$.

For opt-LDMD, we set
\[
r=15,\quad\varepsilon=0.5,
\]
which yields $N=16$ stages. Snapshots are collected from the first $50\%$ of time steps within each stage. The opt-LDMD method thus determines segmentation directly based on the predefined error threshold  $\varepsilon$. This segmentation scheme is subsequently adopted in the configuration of A-LDMD, ensuring that both approaches utilize the same number of snapshots per stage. 

For A-LDMD, we select
\[
r=15,\quad \epsilon=1.5\times10^{-5},
\]
where the residual error threshold $\epsilon$ is tuned to achieve $N=16$, thereby matching the number of stages and prediction rate of opt-LDMD. Additionally, we consider a P-LDMD setting, in which the time domain is uniformly segmented into $N=16$ equal segments, with $50\%$ of the data in each stage used as snapshots.

\begin{figure}[htbp]
\centering
\subfigure{\includegraphics[width=0.8\textwidth]{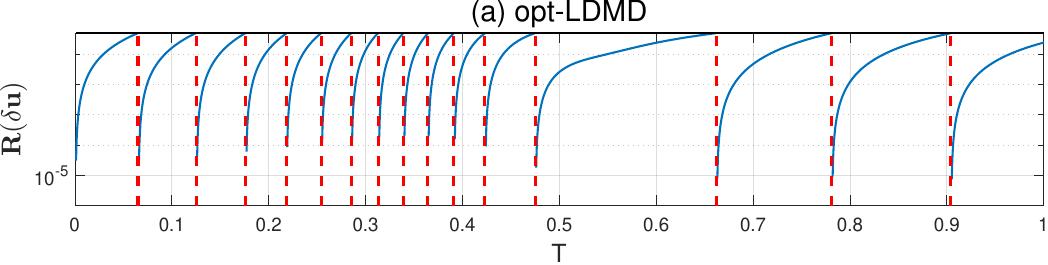}}
\subfigure{\includegraphics[width=0.8\textwidth]{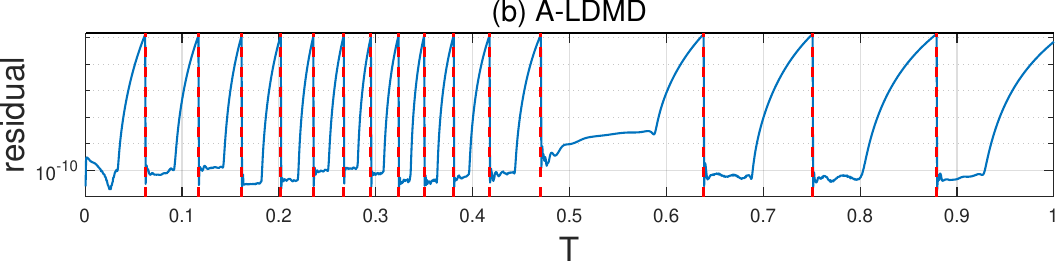}}
\caption{Comparison of remainder and residual evolution for opt-LDMD and A-LDMD: (a) the first-order remainder $\mathbf{R}(\delta\mathbf{u})$ of opt-LDMD and (b) the residuals of A-LDMD. The red dashed lines indicate the locations of the segmentation boundaries.}
\label{fig:Burgers5}
\end{figure}

Figure \ref{fig:Burgers5} demonstrates that opt-LDMD and A-LDMD produce highly similar segmentation structures. In both cases, a higher density of segments is observed during the early time steps, which gradually becomes sparser in the later stages. This trend is consistent with the characteristic evolution of the Burgers’ equation, in which the solution exhibits sharp changes in the initial stage and gradually weakens and stabilizes over time.

\begin{figure}[htbp]
\centering
\subfigure{\includegraphics[width=0.5\textwidth,height=2.0in]{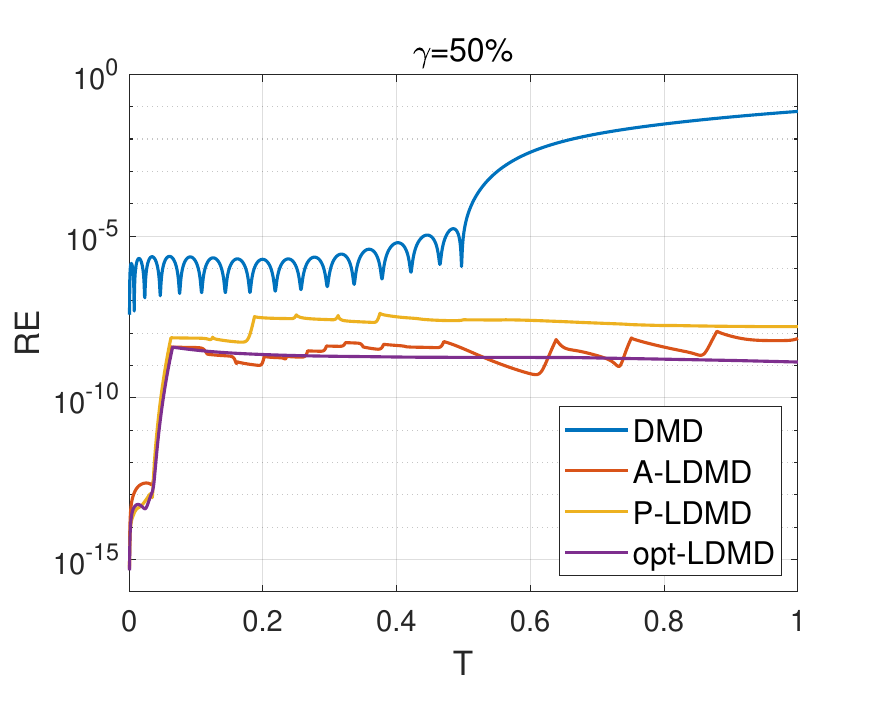}}
\caption{The $L^2$ relative error of the Burgers' equation for DMD, A-LDMD, P-LDMD, and opt-LDMD at $\gamma=50\%$.}
\label{fig:Burgers6}
\end{figure}

\begin{table}[h]
\centering
\caption{Comparison of CPU time and the mean $L^2$ relative error of the Burgers' equation for DMD, P-LDMD, A-LDMD, and opt-LDMD.}
\label{tab:burgers_3}
\setlength{\extrarowheight}{2pt}
\begin{tabular}{c|c}
	\Xhline{1pt}
	Model & MRE \\
	\hline
	DMD & $0.0129$ \\
	\hline
	A-LDMD & $3.2039\times 10^{-9}$ \\
	\hline
	P-LDMD & $1.9413\times 10^{-8}$ \\
	\hline
	opt-LDMD & $1.7486\times10^{-9}$  \\
	\Xhline{1pt}
\end{tabular}
\end{table}

As shown in Figure \ref{fig:Burgers6} and Table \ref{tab:burgers_3}, A-LDMD achieves performance close to that of opt-LDMD. Opt-LDMD maintains stable accuracy after the initial stage, benefiting from its ideal segment-wise linearity based on the full reference solution. Compared with opt-LDMD, P-LDMD exhibits slightly higher errors, while A-LDMD shows marginally larger errors with mild fluctuations. Nevertheless, all LDMD variants consistently outperform standard DMD.


\color{black}
\subsection{Allen-Cahn equation}
We consider the Allen-Cahn equation, a type of reaction-diffusion equation, with Neumann boundary conditions, which has a nonlinear equation source term and is used to describe phase separation and interface motion. It can be expressed as follows:
\begin{equation}\label{eq:numerical_2}
\begin{cases}{\frac{\partial \mathbf{u}}{\partial t}}=\alpha{\frac{\partial^{2}\mathbf{u}}{\partial x^{2}}}+5(\mathbf{u}-\mathbf{u}^3),\quad \mathbf{x}\in[-L,L],\quad t\in[0,T],\\
	\mathbf{u}(0,\mathbf{x})=0.53\mathbf{x} + 0.47\sin(-\frac{3}{2}\pi \mathbf{x}),\\      \left.\frac{\partial \mathbf{u}}{\partial x}\right|_{x=-L}=\left.\frac{\partial \mathbf{u}}{\partial x}\right|_{x=L}=0,\end{cases}
	\end{equation}
	where $L=1$, $T=2$ and the diffusion coefficient $\alpha=10^{-4}$. 
	We discretize the spatial domain into $N_x=200$ intervals and the temporal domain into $N_t=2000$ steps. Then, we set
	\[
	r=15,\quad \epsilon=3\times 10^{-5},\quad n_1=200,\quad m=50.
	\]
	Therefore, {with the number of stages $N=17$}, the prediction rate is $\gamma=50\%$, corresponding to $M=1000$ in the standard DMD method.  Figure \ref{fig:AC1} shows that A-LDMD performs well in approximating the reference solution, whereas the standard DMD method performs relatively poorly in comparison. 
	
	\begin{figure}[htbp]
\centering
\subfigure{\includegraphics[width=1\textwidth]{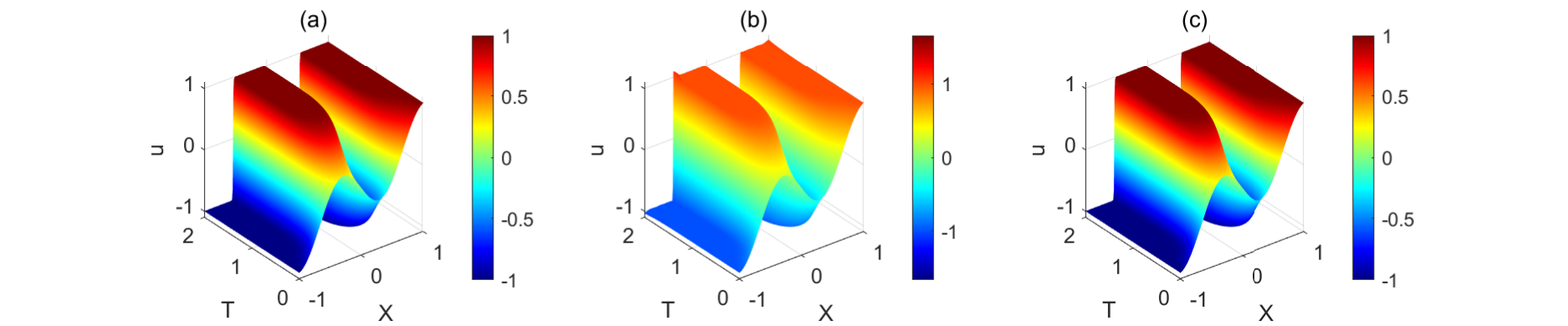}}
\caption{The solution $\mathbf{u}(t,\mathbf{x})$ of the Allen-Cahn equation: (a) the reference solution, (b) the standard DMD solution, and (c) the A-LDMD solution.}
\label{fig:AC1}
\end{figure}

\begin{figure}[htbp]
\centering
\subfigure{\includegraphics[width=0.4\textwidth,height=1.6in]{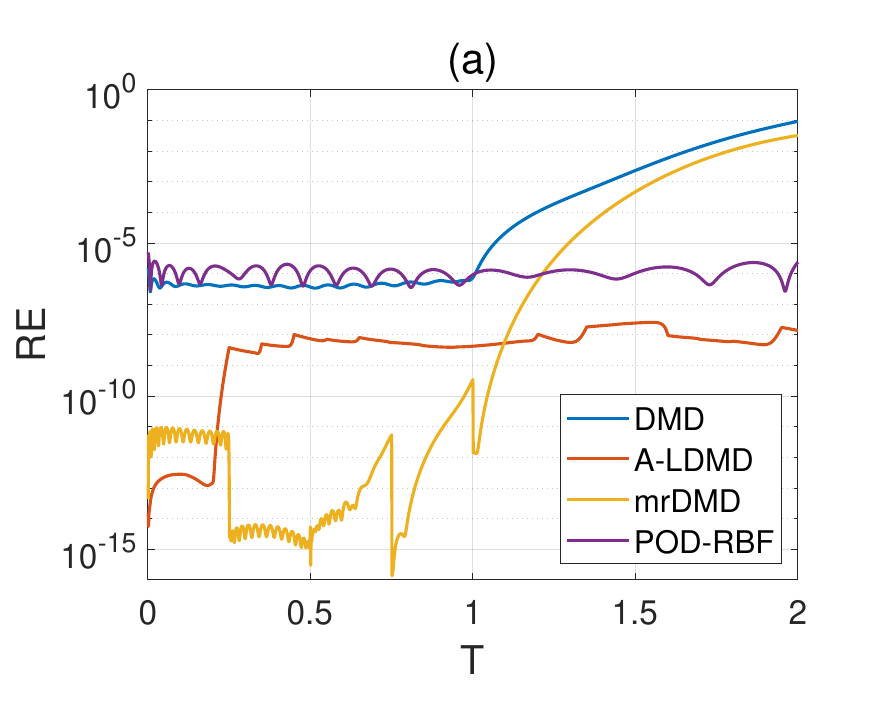}}
\subfigure{\includegraphics[width=0.4\textwidth,height=1.6in]{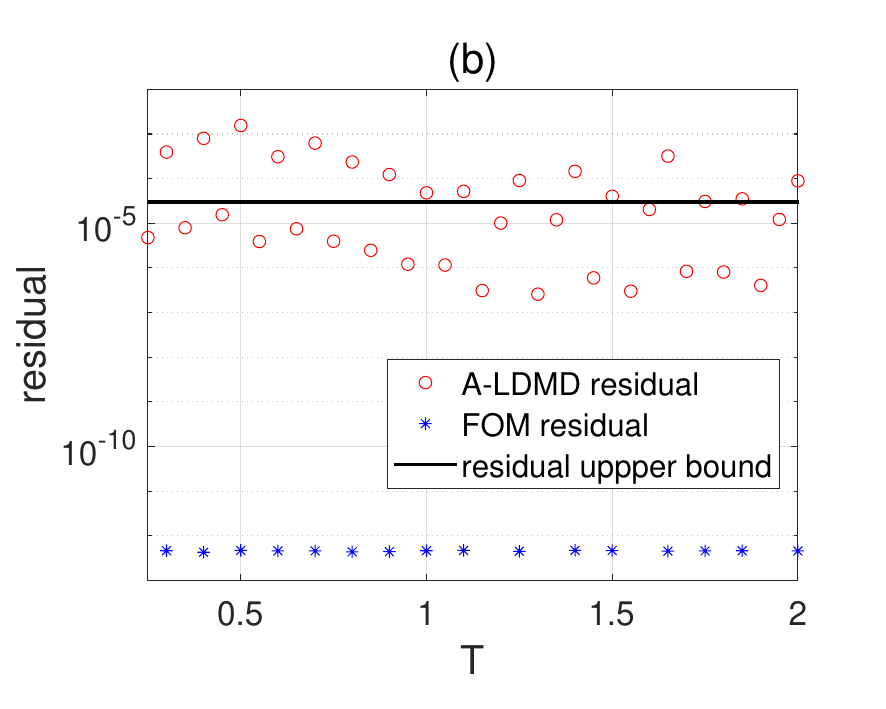}}
\caption{(a) The $L^2$ relative error of the Allen-Cahn equation for DMD, A-LDMD, mrDMD, and POD-RBF at $\gamma=50\%$. (b) Residual of the A-LDMD solution.}
\label{fig:AC2}
\end{figure}

As summarized in Table \ref{tab:AC_1}, A-LDMD delivers the highest predictive accuracy and computational efficiency among all methods. Figure \ref{fig:AC2} (a) shows that the RE of A-LDMD is quickly stable and stays at a low level within the range of $10^{-10}$ to $10^{-7}$. It consistently outperforms the standard DMD method and POD-RBF. The mrDMD demonstrates a similar performance to Figure \ref{fig:Burgers2} (a) due to a lack of robust extrapolation mechanisms, which leads to its prediction result being close to the standard DMD. Moreover, Figure \ref{fig:AC2} (b) reaches a conclusion similar to that in Figure \ref{fig:Burgers2} (b).

\begin{table}[h]
\centering
\caption{Comparison of CPU time and the mean $L^2$ relative error of the Allen-Cahn equation for FOM, DMD, A-LDMD, mrDMD, and POD-RBF at $\gamma=50\%$.}
\label{tab:AC_1}
\setlength{\extrarowheight}{2pt}
\begin{tabular}{c|c|c}
	\Xhline{1pt}
	Model & CPU time(s) & MRE \\
	\hline
	FOM & 52.3481 & / \\
	\hline
	DMD & $26.7059$ & $0.0076$ \\
	\hline
	A-LDMD & $23.0213$ & $7.5795\times 10^{-9}$ \\
	\hline
	mrDMD & $26.8274$ & $0.0027$ \\
	\hline
	POD-RBF & $26.5832$ & $1.1771\times 10^{-6}$ \\
	\Xhline{1pt}
\end{tabular}
\end{table}


\subsection{Nonlinear Schrödinger equation}
We consider one of the fundamental equations in quantum mechanics—the nonlinear Schrödinger equation (NLSE):
\begin{equation}\label{eq:numerical_3}
\begin{cases}\mathbf{\psi}_{t}-i\theta \mathbf{\psi}_{xx}-i\theta|\mathbf{\psi}|^{2}\mathbf{\psi}=0,\quad \mathbf{x}\in[-L,L],\quad t\in[0,T],\\ \mathbf{\psi}(0,\mathbf{x})=2\text{sech}(\mathbf{x}),\\ \mathbf{\psi}(t,-L)=\mathbf{\psi}(t,L)=0,\end{cases}
\end{equation}
where $L=15$, $T=\pi$ and $\theta=0.5$. 
We discretize the spatial domain into $N_x=100$ intervals and the temporal domain into $ N_t=2000$ steps. The solution $\mathbf{\psi}(t,\mathbf{x})$ of the Schrödinger equation is a wave function, and we are concerned about its position density:
\[
\mathbf{\rho}(t,\mathbf{x})=|\mathbf{\psi}(t,\mathbf{x})|^2.
\]
Then, we set
\[
r=10,\quad \epsilon=2\times 10^{-7},\quad n_1=50,\quad m=50.
\]
Therefore, {with the number of stages $N=20$}, the prediction rate is $\gamma=50\%$, corresponding to $M=1000$ in the standard DMD method. Figure \ref{fig:NLSE1} illustrates that the standard DMD method fails to predict the NLSE, whereas A-LDMD accurately approximates the position density of the NLSE. 

\begin{figure}[htbp]
\centering
\subfigure{\includegraphics[width=1\textwidth]{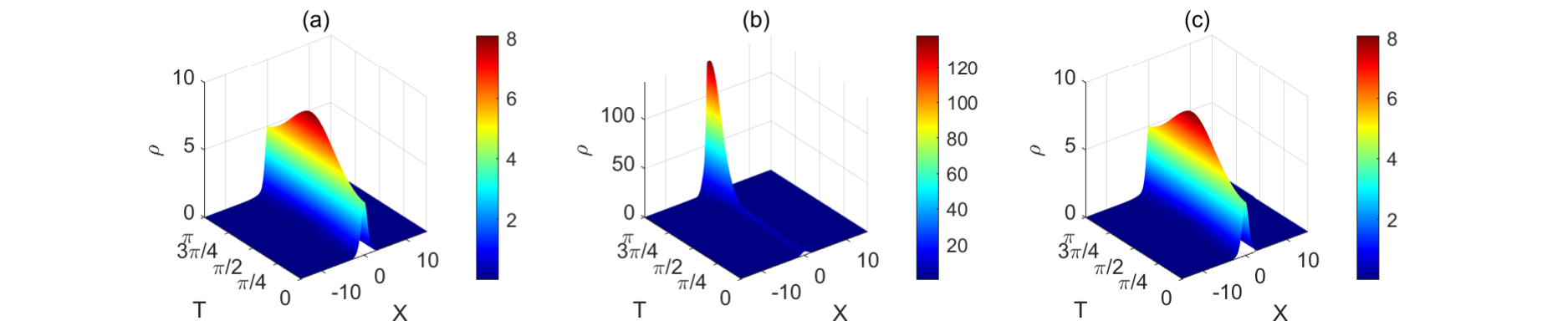}}
\caption{The position density $\mathbf{\rho}(t,\mathbf{x})$ of the NLSE: (a) the reference position density, (b) the standard DMD position density, and (c) the A-LDMD position density.}
\label{fig:NLSE1}
\end{figure}

\begin{figure}[htbp]
\centering
\subfigure{\includegraphics[width=0.4\textwidth,height=1.6in]{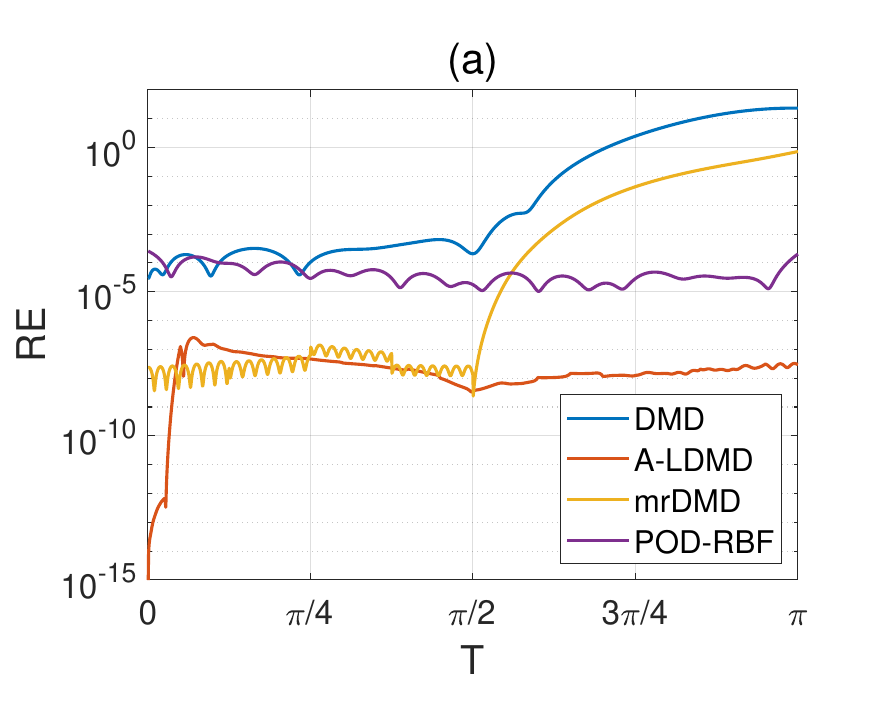}}
\subfigure{\includegraphics[width=0.4\textwidth,height=1.6in]{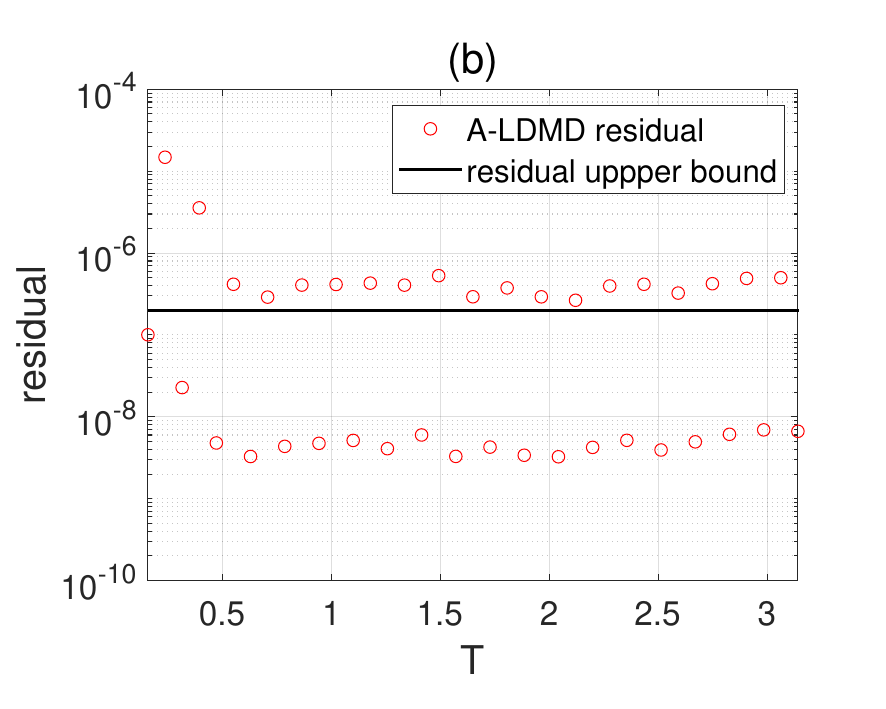}}
\caption{(a) The $L^2$ relative error of the NLSE for DMD, A-LDMD, mrDMD, and POD-RBF at $\gamma=50\%$. (b) Residual of the A-LDMD solution.}
\label{fig:NLSE2}
\end{figure}

Figure \ref{fig:NLSE2} (a) demonstrates that A-LDMD always outperforms the standard DMD method and POD-RBF. Due to the oscillatory nature of the NLSE, the standard DMD method yields unsatisfactory predictions even when appropriate observation functions, larger SVD truncation numbers, and additional snapshot data are used. Through a reasonable choice of residual upper bound and time window size, our method maintains a stable prediction error between $10^{-7}$ and $10^{-8}$, achieving the highest computational efficiency and superior prediction accuracy. The reconstruction error of mrDMD is almost comparable to the subsequent prediction error of A-LDMD, while it exhibits rapid error growth, so that fails to deliver satisfactory predictions for the NLSE. The prediction error of POD-RBF remains relatively stable, fluctuating between $10^{-5}$ and $10^{-4}$.


The residual fluctuates around the prescribed residual error threshold. Figure \ref{fig:NLSE2} (b) indicates that the DMD prediction and the FOM correction are alternated. Different from Figure \ref{fig:Burgers2} (b) and Figure \ref{fig:AC2} (b), the residual equation here corresponds directly to the FOM, which eliminates numerical errors when the FOM is used to correct the prediction results.

\begin{table}[h]
\centering
\caption{Comparison of CPU time and the mean $L^2$ relative error of the NLSE for FOM, DMD, A-LDMD, mrDMD, and POD-RBF at $\gamma=50\%$.}
\label{tab:NLSE_1}
\setlength{\extrarowheight}{2pt}
\begin{tabular}{c|c|c}
	\Xhline{1pt}
	Model & CPU time(s) & MRE \\
	\hline
	FOM & 0.1333 & / \\
	\hline
	DMD & 0.0695 & $3.5012$ \\
	\hline
	A-LDMD & 0.0921 & $3.2403\times 10^{-8}$ \\ 
	\hline
	mrDMD & 0.1867 & $0.0375$ \\
	\hline
	POD-RBF & 0.1127 & $4.9948\times 10^{-5}$ \\
	\Xhline{1pt}
\end{tabular}
\end{table}

Table \ref{tab:NLSE_1} concludes similarly to that in Table \ref{tab:burgers_4}. Due to its relatively complex processing pipeline, mrDMD incurs even higher computational cost than the FOM.


\subsection{Maxwell's equations}
Finally, we consider the time-domain Maxwell's equations, which are coupled systems used to describe the classical electromagnetic equations \cite{li2012time}:
\begin{equation}\label{eq:numerical_4}
\begin{cases}\frac{\partial \mathbf{H}}{\partial t}=-\mathbf{curl}~\mathbf{E}-\mathbf{K}+\mathbf{g}^s,&\mathbf{x}\in\Omega ,\quad t\in (0,T],\\
	\frac{\partial \mathbf{E}}{\partial t}=\mathbf{curl}~\mathbf{H}-\mathbf{J}+\mathbf{f}^s,&\mathbf{x}\in\Omega ,\quad t\in (0,T],\\
	\frac{\partial \mathbf{J}}{\partial t}=\mathbf{E}-\mathbf{J},&\mathbf{x}\in\Omega ,\quad t\in (0,T],\\
	\frac{\partial \mathbf{K}}{\partial t}=\mathbf{H}-\mathbf{K},&\mathbf{x}\in\Omega ,\quad t\in (0,T],\end{cases}
	\end{equation}
	where $\mathbf{H}=(H_x,H_y,H_z)^T$ is the magnetic field, $\mathbf{E}=(E_x,E_y,E_z)^T$ is the electric field, $\mathbf{J}=(J_x,J_y,J_z)^T$ is the polarization current density, and $\mathbf{K}=(K_x,K_y,K_z)^{T}$ is the magnetization current density. In this case, we consider the 2-D time-domain case of transverse magnetic (TM) mode formulation, characterized by the absence of magnetic field components in the $z$-direction, i.e. $\mathbf{H}=(H_x,H_y)^T$, $\mathbf{E}=E_z$, $\mathbf{J}=J_z$, and $\mathbf{K}=(K_x,K_y)^{T}$; the differential operators are
	\begin{equation*}
\mathbf{curl}~E=(\frac{\partial}{\partial y}E,-\frac{\partial}{\partial x}E)^T,\quad\mathrm{curl}~\mathbf{H}=\frac{\partial}{\partial x}H_y-\frac{\partial}{\partial y}H_x.
\end{equation*}
The source term
\begin{equation*}
\mathbf{f}^s(t,\mathbf{x})=(t-1-2\omega)\sin(\omega x)\sin(\omega y)e^{-t},
\end{equation*}
while $\mathbf{g}^s=(g^s_x,g^s_y)^{T}$ 
\begin{align*}
g^s_x(t,\mathbf{x})=(\omega-1+t)\sin(\omega x)\cos(\omega y)e^{-t},\\
g^s_y(t,\mathbf{x})=(1-\omega-t)\cos(\omega x)\sin(\omega y)e^{-t}.
\end{align*}
The initial conditions are given by
\begin{align*}
&\mathbf{H}(0,\mathbf{x})=(\sin(\omega x)\cos(\omega y),-\cos(\omega x)\sin(\omega y))^{T},\\
&\mathbf{E}(0,\mathbf{x})=\sin{\omega x}\cos{\omega y},\\
&\mathbf{J}(0,\mathbf{x})=0,\quad\quad\mathbf{K}(0,\mathbf{x})=0,
\end{align*}
where $\omega=4\pi$. The spatial domain $\Omega=[0,1]^2$, the time domain $T=2$. We performed a triangulation over the spatial domain, discretized uniformly into $20$ nodes in each direction, while the temporal domain was equally discretized into $N_t=2000$ steps. We choose the observable function $\mathbf{g}(\mathbf{u})=\left[\mathbf{u},e^\mathbf{u}\right]^{T}$. Due to the difficulty of the residual equation in coupled systems, we set
\[
r=15,\quad n_i=\begin{cases}
90, & i = 1, \\
50, & i \ge 2,
\end{cases}~ \quad m_i=\begin{cases}
10, & i = 1, \\
50, & i \ge 2.
\end{cases}~
\]
Therefore, {with the number of stages $N=20$}, the prediction rate is $\gamma=48\%\approx50\%$, corresponding to $M=1040$ in the standard DMD method. 


Due to the rapid variations in the variables of Maxwell's equations during the initial time phase, an insufficient number of snapshots or an excessively long prediction time step can lead to a sharp increase in error. When using P-LDMD, prior information is often limited. However, we can leverage the flexibility of our segmentation strategy to employ a larger number of snapshots and a smaller prediction time step in the initial phase, thereby effectively controlling error growth and providing a more accurate initial condition for the subsequent FOM correction.

Figure \ref{fig:Maxwell11}, \ref{fig:Maxwell13}, and \ref{fig:Maxwell15} show that the prediction error of the standard DMD method is substantial. In contrast, P-LDMD yields considerably more accurate prediction results. Since Maxwell’s equations (\ref{eq:numerical_4}) are intrinsically coupled under time-varying dynamics, employing a single linear operator to describe the overall system cannot explicitly capture the interaction between subsystems.

\begin{figure}[htbp]
\centering
\subfigure{\includegraphics[width=0.9\textwidth]{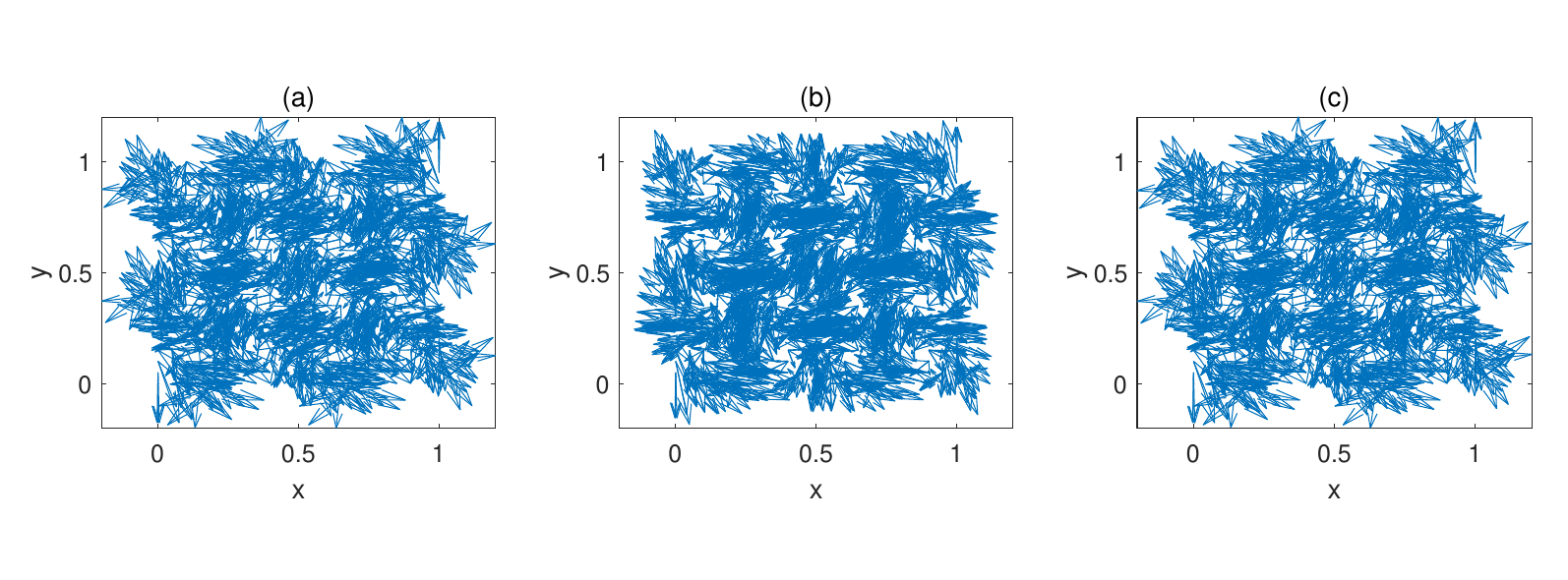}}
\caption{The magnetic field $\mathbf{H}$ of Maxwell's equations at $T=2$: (a) the reference solution, (b) the standard DMD solution, and (c) the P-LDMD solution.}
\label{fig:Maxwell11}
\end{figure}

\begin{figure}[htbp]
\centering
\subfigure{\includegraphics[width=1\textwidth]{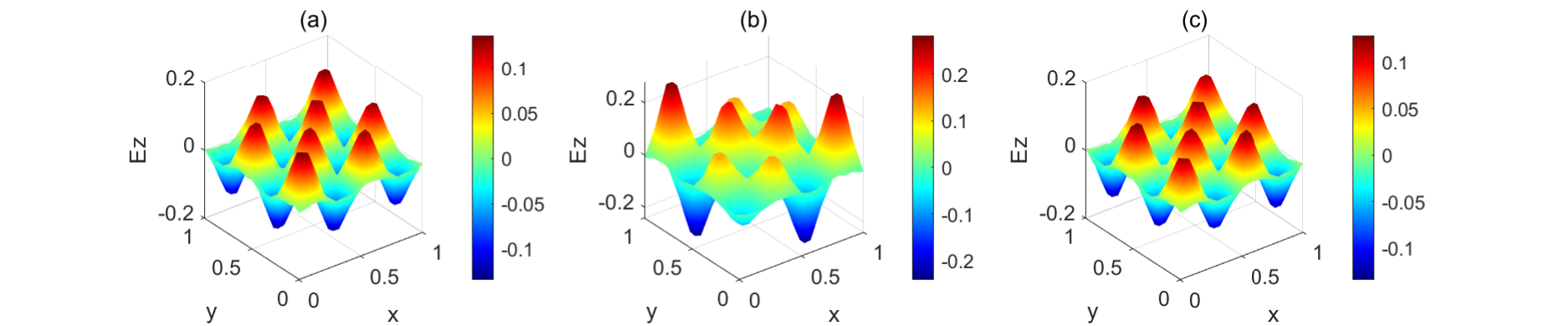}}
\caption{The electric field $E_z$ of Maxwell's equations at $T=2$: (a) the reference solution, (b) the standard DMD solution, and (c) the P-LDMD solution.}
\label{fig:Maxwell13}
\end{figure}

\begin{figure}[htbp]
\centering
\subfigure{\includegraphics[width=1\textwidth]{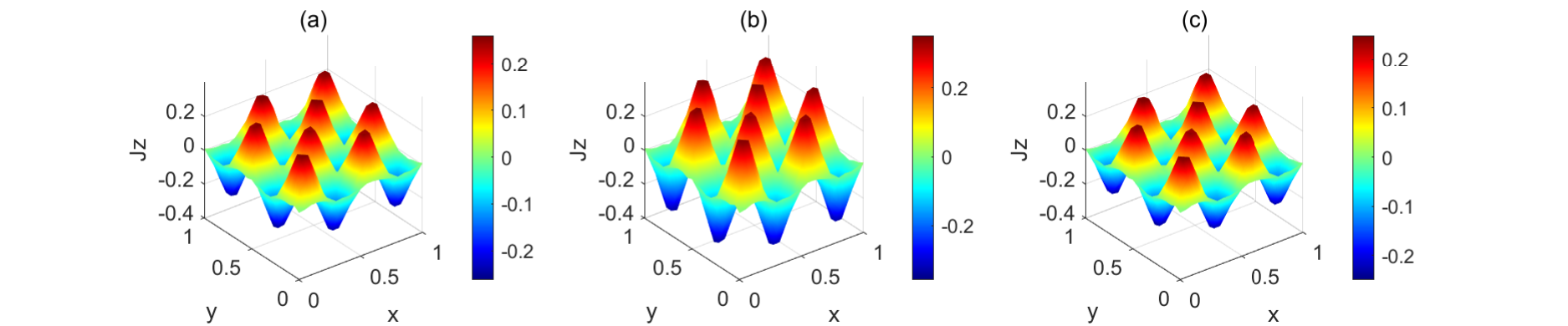}}
\caption{The polarization current density $J_z$ of Maxwell's equations at $T=2$: (a) the reference solution, (b) the standard DMD solution, and (c) the P-LDMD solution.}
\label{fig:Maxwell15}
\end{figure}

In this case, we also compare P-LDMD with two representative approaches: POD-RBF and a leading time-dependent DMD variant. Owing to the unsatisfactory performance of mrDMD, HODMD is adopted as an alternative for comparison. In the HODMD formulation, $d$ designates the number of time delays.


For the 2-D TM case, $H_x$ and $H_y$ have certain symmetry and expressions. Therefore, the DMD prediction results for both are identical. We only exhibit $H_y$ in the numerical experiment, and $H_x$ has the same prediction results.

\begin{figure}[htbp]
\centering
\subfigure{\includegraphics[width=0.32\textwidth,height=1.28in]{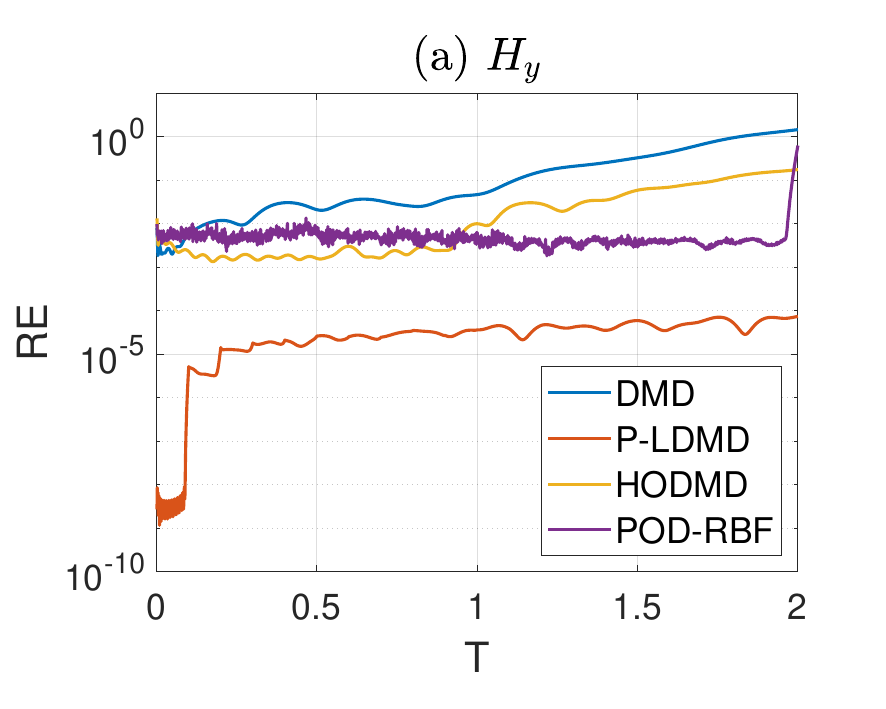}}
\subfigure{\includegraphics[width=0.32\textwidth,height=1.28in]{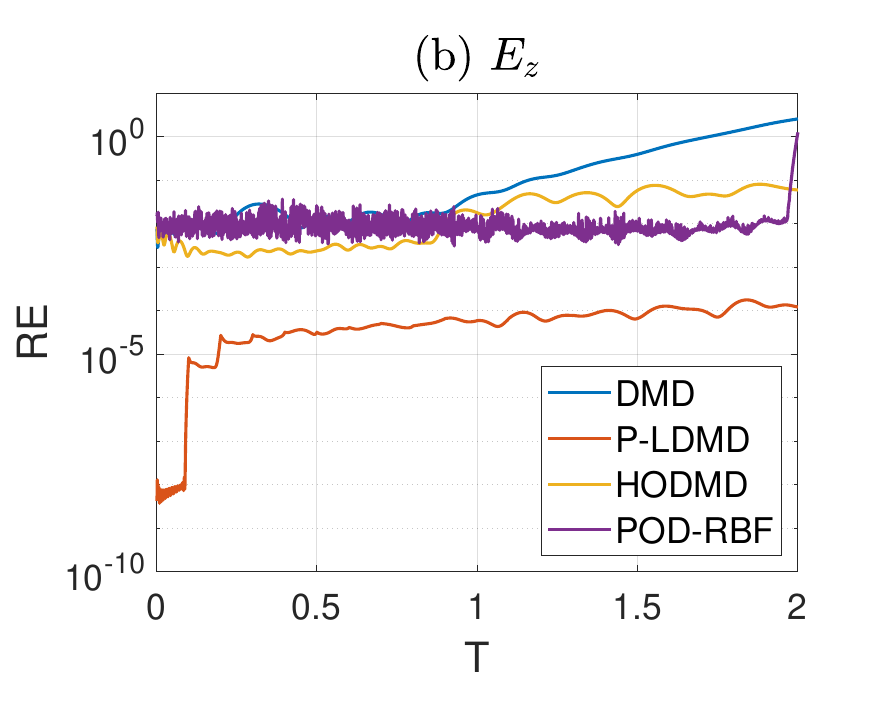}}
\subfigure{\includegraphics[width=0.32\textwidth,height=1.28in]{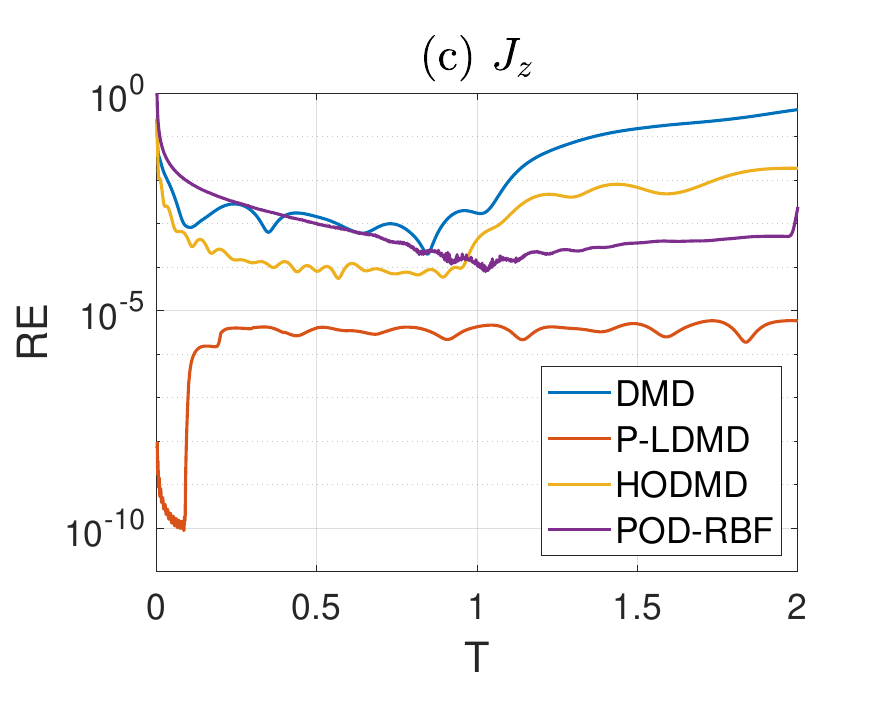}}
\caption{Model $L^2$ relative error of Maxwell's equations for DMD, P-LDMD, HODMD, and POD-RBF at $\gamma=48\%$: (a) magnetic field $H_y$, (b) electric field $E_z$, and (c) polarization current density $J_z$.}
\label{fig:Maxwell1}
\end{figure}

As presented in Table \ref{tab:Maxwell}, P-LDMD achieves superior predictive performance and computational efficiency compared to all other methods. In Figure \ref{fig:Maxwell1} (a) and (b), the prediction error of POD-RBF fluctuates around $10^{-2}$ but increases significantly in the final time steps. Figure \ref{fig:Maxwell1} (c) reveals that a relatively large prediction error occurs initially, and then gradually decreases and stabilizes between $10^{-4}$ and $10^{-3}$. Additionally, HODMD slightly outperforms standard DMD.

\begin{table}[hbtp]
\centering
\label{tab:Maxwell}
\caption{Comparison of CPU time and the mean $L^2$ relative error of Maxwell's equations for FOM, DMD, P-LDMD, HODMD, and POD-RBF at the prediction rate $\gamma=48\%$.}
\setlength{\extrarowheight}{2pt}
\begin{tabular}{c|c|c|c}
	\Xhline{1pt}
	Model & CPU time(s) & \multicolumn{2}{c}{MRE} \\ 
	\hline
	FOM & 294.5803 & \multicolumn{2}{c}{/} \\
	\hline
	\multirow{3}{*}{\centering DMD} & \multirow{3}{*}{\centering 151.0478} & magnetic field & 0.2614 \\
	
	\cline{3-4}
	& & electric field & 0.3594 \\
	
	\cline{3-4}
	& & polarization current density & 0.0801 \\
	
	\hline
	\multirow{3}{*}{\centering P-LDMD} & \multirow{3}{*}{\centering 148.0742} & magnetic field & $3.3659\times 10^{-5}$ \\
	
	\cline{3-4}
	& & electric field & $6.4252\times 10^{-5}$ \\
	
	\cline{3-4}
	& & polarization current density & $3.5211\times 10^{-6}$ \\
	
	\hline
	\multirow{3}{*}{\centering HODMD} & 151.8481($d=100$) & magnetic field & 0.0355 \\
	
	\cline{2-4}
	& 152.8023($d=150$) & electric field & 0.0265 \\
	
	\cline{2-4}
	& 151.5067($d=80$) & polarization current density & 0.0043 \\
	
	\hline
	\multirow{3}{*}{\centering POD-RBF} & \multirow{3}{*}{\centering 151.4912} & magnetic field & 0.0076 \\
	
	\cline{3-4}
	& & electric field & 0.0148 \\
	
	\cline{3-4}
	& & polarization current density & 0.0034 \\
	
	\Xhline{1pt}
	
\end{tabular}
\end{table}

\section{Conclusions}\label{sec:conclu}
In this work, we have proposed a novel localized variant of DMD that integrates the standard DMD framework with a time-domain decomposition strategy. Leveraging the concept of the linearized truncation of the Taylor expansion, the temporal domain is divided into consecutive short segments exhibiting strong local linearity, thereby maximizing the use of DMD’s intrinsic linear predictive capability.
Concerning the temporal segmentation, we first proposed LDMD with predefined segmentation, which is straightforward to implement but relies on prior knowledge of the system to achieve higher accuracy. To overcome this limitation, we further developed LDMD with adaptive segmentation, which employs an error estimator for control and operates without prior system observations, thereby significantly enhancing its practical applicability. 


Numerical results demonstrate that LDMD achieves the highest numerical accuracy and maintains superior computational efficiency compared to time-dependent DMD variants such as mrDMD and HODMD. Moreover, since the low-rank approximation core of DMD relies on SVD truncation, LDMD employs a localized strategy, meaning the matrix dimension for SVD at each stage is significantly smaller than that of standard DMD. This may lead to a higher computational efficiency than standard DMD.  Additionally, the residual computation in our approach can be performed on a coarser spatial grid, thereby further reducing computational cost. 


For future research, we aim to develop a more effective error estimator that aligns with the equation-free nature of DMD, ensuring broader applicability, computational simplicity, and the ability to correct inaccurate predictions for subsequent stages. Additionally, we seek to provide a comprehensive and rigorous theoretical analysis comparing LDMD with standard DMD, demonstrating that LDMD consistently outperforms standard DMD across all time intervals by optimizing the time window size and error estimator threshold. {Furthermore, we will consider combining the LDMD framework with the ResDMD \cite{colbrook2024rigorous} methodology to enhance predictive capability in chaotic systems exhibiting extreme sensitivity to initial conditions and inherent long-term unpredictability.} Ultimately, we aspire to establish LDMD as a high-accuracy surrogate model for parametric dynamical systems, further enhancing the predictive capabilities of existing methodologies \cite{model2024Song,adaptive2024li}.

\section*{Acknowledgments}
The authors would like to express their sincere appreciation to the anonymous reviewers for their insightful comments and constructive suggestions, which have greatly contributed to improving the quality and clarity of this manuscript.

\bibliographystyle{siamplain}

\end{document}